\documentclass[10pt]{article}

\usepackage{amsmath}
\usepackage{amssymb}
\usepackage{amsthm}
\usepackage[usenames,dvipsnames]{pstricks}
\usepackage{epsfig}
\usepackage{enumerate}

\newcounter{zhipengcount}
\newcounter{jinhocount}
\newtheorem{thm}{Theorem}[section]

\newtheorem{lm}{Lemma}[section]
\newtheorem{prop}{Proposition}[section]

\newtheorem{rmk}{Remark}[section]
\newtheorem{cor}{Corollary}[section]

\newcommand{\widthf}{W}
\newcommand{\width}{M}
\newcommand{\Prob}{\mathbb{P}}
\newcommand{\Z}{\mathbb{Z}}

\newcommand{\R}{\mathbb{R}}
\newcommand{\C}{\mathbb{C}}

\newcommand{\ii}{i}

\newcommand{\s}{s}
\newcommand{\x}{x}

\newcommand{\iin}{\mathrm{in}}
\newcommand{\oou}{\mathrm{out}}

\newcommand{\nodes}{\mathcal{D}}
\newcommand{\idn}{\mathbb{I}}

\newcommand{\Ne}{\mathcal{N}}

\newcommand{\Ha}{\mathbf{H}}
\newcommand{\HaC}{\mathcal{H}}
\newcommand{\To}{\mathbf{T}}
\newcommand{\ToC}{\mathcal{T}}
\newcommand{\al}{d}
\newcommand{\calL}{\mathcal{L}}
\newcommand{\N}{\mathbb{N}}

\newcommand{\lift}{\mathbb{L}}

\author{Jinho Baik\footnote{Department of Mathematics, University of Michigan,
Ann Arbor, MI, 48109, USA \newline email: \texttt{baik@umich.edu}}
and Zhipeng Liu\footnote{Department of Mathematics, University of
Michigan, Ann Arbor, MI, 48109, USA
\newline email: \texttt{zhpliu@umich.edu}}}
\date{\today}

\begin{document}
\title{Discrete
Toeplitz/Hankel determinants
and the
width of non-intersecting processes
}
\maketitle

\begin{abstract}
We show that the ratio of a discrete Toeplitz/Hankel determinant and its continuous counterpart equals a Freholm determinant involving continuous orthogonal polynomials.
This identity is used to evaluate a triple asymptotic of some discrete Toeplitz/Hankel determinants which arise in studying non-intersecting processes.
We show that the asymptotic fluctuations of the width of such processes are given by the GUE Tracy-Widom distribution.
This result leads us to an identity between the GUE Tracy-Widom distribution and the maximum of the sum of two independent Airy processes minus a parabola. We provide an independent proof of this identity.
\end{abstract}

\section{Introduction}\label{sec:int}

This paper consists of two parts. 
First, we develop a general method for an asymptotic analysis of the Toeplitz or Hankel determinants of \emph{discrete} measure using orthogonal polynomials with respect to a \emph{continuous} measure.
In the second part, which is longer, we evaluate the limiting distribution of the ``width'' of non-intersecting processes as an application.
This leads to the discovery of an interesting identity between the GUE Tracy-Widom distribution and the maximum of the sum of two independent Airy processes minus a parabola: see Theroem~\ref{z11-3-4}.

\subsection{Discrete Toeplitz determinants}

For a finite subset $\nodes$
 of the unit circle $\Sigma$ in the complex plane and a function $f:\nodes\to \R$,
 the Toeplitz determinant of the discrete measure $\frac1{|\nodes|} \sum_{z\in \nodes} f(z) \delta_z$ is defined as
\begin{equation}
	\To_n(f,\nodes)= \det \bigg[ \frac1{|\nodes|} \sum_{z\in \nodes} z^{-j+k} f(z) \bigg]_{j,k=0}^{n-1}.
\end{equation}
Since the Cauchy-Viennet/Andreief's formula implies that
\begin{equation}
	\To_n(f,\nodes)= \frac1{n!|\nodes|^n} \sum_{(z_1, \cdots, z_n)\in \nodes^n} \prod_{1\le j<k\le n} |z_j-z_k|^2 \prod_{j=1}^n f(z_j),
\end{equation}
this can also be thought of as the partition function of the discrete Coulomb gases with potential $f$ where the charges are confined to be on the discrete set $\nodes$.
Note that $\To_n(f,\nodes)=0$ unless $n\le |\nodes|$.

The Toepltiz and Hankel determinants of discrete measures also arise in many other problems. A few examples are
\begin{enumerate}
\item joint distribution of maximal crossing and maximal nesting of random matchings \cite{CDDSY, BaikJenkins}
\item maximal height of non-intersecting excursions on the half-line \cite{KobayashiIzumiKatori08, SchehrMajumdaretal08, Forrester10, Liechty11}
\item periodic totally asymmetric simple exclusion process \cite{LiuUpcoming}, and
\item width of non-intersecting
processes: see section~\ref{sec:withNiBm} below.
\end{enumerate}

For a continuous function $f$ on the unit circle, the usual Toeplitz determinant of the continuous symbol $f$ is defined as
\begin{equation}
	\ToC_n(f)= \det \bigg[ \int_{|z|=1} z^{-j+k} f(z) \frac{dz}{2\pi iz} \bigg]_{j,k=0}^{n-1}.
\end{equation}
For convenience, we call this Toeplitz determinant  \emph{continuous Toeplitz determinant}, and the Toeplitz determinant with a discrete measure  \emph{discrete Toeplitz determinant}. They are denoted by $\ToC_n$ and $\To_n$ respectively.

A discrete Toeplitz determinant contains parameters
\begin{enumerate}[(i)]
\item $n$, the size of the matrix,
\item 
$m=|\nodes|$, the size of the discrete set, and
\item $t$, a parameter of the function $f$.
\end{enumerate}
It is often of interest to study the asymptotics of $\To_n(f,\nodes)$ as all or some of the parameters become large.
From the Coulomb gas interpretation, we see that the discrete set imposes the minimal distance between the Coulumb charges.
This leads one to the doubly constrained equilibrium measure problem of finding a probability measure $\mu$ such that $0\le \mu \le \lim_{|\nodes|\to \infty} \frac1{|\nodes|} \sum_{z\in \nodes} \delta_z$.
Note that the upper constraint is absent for the related equilibrium measure problem for continuous Toeplitz determinants.
One way to evaluate the asymptotics of discrete Toeplitz determinants rigorously is to use the discrete orthogonal polynomials. A Deift-Zhou steepest-descent method \cite{DeiftZhou, DKMVZ3, DeiftBook} for the Riemann-Hilbert problems of general discrete orthogonal polynomials was previously developed in \cite{BKMMbook}.

The observation of this paper is that it is possible to study the asymptotics using \underline{continuous} orthogonal polynomials instead. This follows from a simple identity.
To state this identity, let $\Sigma$ be the positively-oriented unit circle and we assume the followings:
\begin{itemize}
\item[(a)] Let $\nodes$ be a finite discrete subset of $\Sigma$ and let $\Omega$ be a neighborhood of $\Sigma$.
\item[(b)] Let $f(z)$ be a non-trivial analytic function on $\Omega$ such that $f(z)\ge0$ for all $z\in\Sigma$.
\end{itemize}
Let $p_k(z)=\kappa_kz^k+\cdots$ be the orthonormal  polynomials with the continuous measure $f(z)\frac{dz}{2\pi iz}$ on the unit circle.
The `reversed polynomials' are defined by $p_k^*(z):= z^k \overline{p_k(\bar{z}^{-1})}$.
Let $\gamma(z)$ be an analytic function on $\Omega$ such that $\gamma$ vanishes exactly on $\nodes$ and all the zeros are simple. There are such functions from complex analysis.

\begin{thm}
\label{thm:0}
Assuming (a), (b) above, we have
\begin{equation}
  \label{z3-17-1}
  	\To_n(f, \nodes) =\ToC_n(f)\det\left(1+K\right)_{L^2(\Sigma_{in} \cup \Sigma_{out},\frac{dz}{2\pi iz})},
\end{equation}
where $K$ is the integral operator with kernel
\begin{equation}\label{eq:KforToe}
    K(z,w) = K_{conti}(z,w)
	\sqrt{v(z)v(w)f(z)f(w)}.
\end{equation}
with
\begin{equation}\label{eq:Kconti}
    K_{conti}(z,w):= z^{-n} \frac{p_n(z)p^*_{n}(w)-p^*_{n}(z)p_{n}(w)}{1-z^{-1}w}.
\end{equation}
Here the contours $\Sigma_{in}$ and $\Sigma_{out}$ are positively-oriented circles of radii $1-\epsilon$ and $1+\epsilon$, respectively, for a small $\epsilon>0$,
and
\begin{equation}\label{eq:vformu}
  	v(z):=\begin{cases}
    -\frac{z\gamma'(z)}{|\nodes|\gamma(z)},\quad &z\in \Sigma_{in},\\
    \frac{z\gamma'(z)}{|\nodes| \gamma(z)}-1 ,\quad &z\in \Sigma_{out}.
  \end{cases}
\end{equation}
\end{thm}

\begin{rmk}
Recall that the Christoffel-Darboux kernel for the orthogonal polynomials on the unit circle is
\begin{equation}\label{eq:CDformu}
  	K_{CD}(z,w)= \sum_{k=0}^{n-1} \overline{p_k(z)} p_k(w)
	= \frac{\overline{p^*_{n}(z)} p^*_n(w) - \overline{p_n(z)}p_n(w)}{1-\bar{z}w}.
\end{equation}
The kernel in~\eqref{eq:Kconti} satisfies $K_{conti}(z,w)= K_{CD}(\bar{z}^{-1},w)$.
\end{rmk}

Note that only the term $v(z)v(w)$ depends on the discrete set $\nodes$ on the right-hand-side of~\eqref{z3-17-1}.

As a special case, when $\nodes=\{z: z^m=1\}$,  we can take $\gamma(z)= z^m-1$.
In this case,
\begin{equation}\label{eq:vforCor}
  	v(z):=\begin{cases}
    	\frac{z^m}{1-z^m} ,&z\in \Sigma_{in},\\
   	\frac{z^{-m}}{1-z^{-m}} ,&z\in \Sigma_{out}.
  \end{cases}
\end{equation}
Observe that $v(z)$ decays exponentially on $\Sigma_{in}$ and $\Sigma_{out}$. From this we can derive the following result when $f$ is fixed and $m$ and $n$ tend to infinity 
easily. See Section~\ref{sec:THf} for the proof. Note that if $n$ is fixed, then the result holds trivially.



\begin{cor}\label{cor:00}
Let $f$ satisfy the assumptions of Theorem~\ref{thm:0} and we assume that $f(z)>0$ for all $|z|=1$.
Let $\nodes=\nodes_m=\{z\in\C : z^m=1\}$. Then there is a positive constant $c$ such that
\begin{equation}
	\To_n(f, \nodes_m)= \ToC_n(f)(1+O(e^{-c(m-n)})
\end{equation}
as $m-n\to \infty$ and $n\to \infty$.
\end{cor}

In many applications we are interested in the ratio $\To_n(f, \nodes_m)/\ToC_n(f)$ where $f$ depends on  $n$  and another parameter, say $t$, in the limit as $m,n, t\to \infty$.
An advantage of using the formula~\eqref{z3-17-1} over the Toeplitz determinants is that
one may be able to find the asymptotic of the ratio even if it is not easy to obtain the asymptotics of the Toeplitz determinants themselves.
See Remark~\ref{rmk:sec4} in Section~\ref{sec:SRW}.

\medskip

We also consider discrete Hankel determinants. Let $\nodes$ be a discrete subset of $\R$. For a function $f$ on $\nodes$, we denote by
\begin{equation}
	\Ha_n(f, \nodes)= \det\bigg[ \sum_{x\in \nodes} x^{j+k} f(x) \bigg]_{j,k=0}^n
\end{equation}
the discrete Hankel determinant. For a function $f$ on $\R$, the continuous Hankel determinant is
\begin{equation}
	\HaC_n(f)= \det\bigg[ \int_{\R} x^{j+k} f(x) dx \bigg]_{j,k=0}^n.
\end{equation}
See Theorem~\ref{thm:2} for an analogue of Theorem~\ref{thm:0} in the Hankel setting.
In the next subsection, we use this theorem to study non-intersecting processes.

\subsection{Width of non-intersecting Brownian bridges}\label{sec:withNiBm}

The non-intersecting processes have been studied extensively in relation to random matrix theory, directed polymers, and random tilings (see, e.g., \cite{Forrester01, Baik00,Johansson03,TracyWidom07}). In this paper, we consider the `width' of three processes. 
We discuss the results on the Brownian bridges in this section. Symmetric simple random walks in both continuous time and discrete time are considered in Section~\ref{sec:SRW}.

Let $X_i(t)$, $i=1, \cdots, n$, be independent standard Brownian motions conditioned that $X_1(t)<X_2(t)< \cdots< X_n(t)$ for all $t\in (0,1)$  and $X_i(0)=X_i(1)=0$ for all $i=1, \cdots, n$.
The \emph{width} is defined as
\begin{equation}\label{eq:1}
	\widthf_n:=\sup_{0\le t\le 1} \left(X_{n}(t)- X_1(t) \right).
\end{equation}
Note that the event that $W_n< M$ equals the event that the Brownian motions stay in the chamber
$x_1<x_2<\cdots<x_n<x_1+M$ for all $t\in (0,1)$.
An application of  the Karlin-McGregor argument in the chamber \cite{KM, HobsonWerner} implies the following formula. See Section~\ref{sec:HBMfo} for the proof.

\begin{prop}\label{prop:BMHank}
Let $\widthf_n$ be defined in~\eqref{eq:1}. Then
\begin{equation}
\label{eq:4}
    \Prob\left(\widthf_n< \width\right) =
    \frac{\left(\frac{\sqrt{2}\pi}{\width\sqrt{n}}\right)^n}{ \HaC_n(F)} \int_0^1      \Ha_n(F, \nodes_s)
    d\s,\qquad F(x)=e^{-nx^2},
\end{equation}
where
\begin{equation}
    \nodes_{\s}:=\bigg\{\frac{\sqrt{2}\pi}{\width\sqrt{n}}(m-{\s}): m\in\Z\bigg\}.
 \end{equation}
\end{prop}

From the Hankel analogue of Theorem~\ref{thm:0}, the asymptotics of the above probability can be studied by using the orthogonal polynomials with respect to $e^{-nx^2}$, i.e. Hermite polynomials. We obtain:

\begin{thm}
\label{thm:1}
Let $\widthf_n$ be the width of $n$ non-intersecting Brownian bridges with duration $1$ given in~\eqref{eq:1}. Then for every $\x\in\R$,
\begin{equation}\label{eq:thm1formula}
    \lim_{n\to\infty}
    \Prob\left( (\widthf_n-2\sqrt{n}) 2^{2/3}n^{1/6}
    \le \x \right)
    = F(\x)
\end{equation}
where $F(\x)$ is the GUE Tracy-Widom distribution function \cite{TW94}.
\end{thm}

\begin{rmk}
The discrete Hankel determinant $\Ha_n(F, \nodes_0)$ with $s=0$ was also appeared in \cite{Forrester10} (see Model I and the equation (14), which is given in terms of a multiple sum) in the context of a certain normalized reunion probability of non-intersecting Brownian motions with periodic boundary condition.
In the same paper, a heuristic argument that a double scaling limit is $F(x)$
was discussed.
Nevertheless, the interpretation in terms of the width of non-intersecting Brownian motions
and a rigorous asymptotic analysis were not given in \cite{Forrester10}.
\end{rmk}


Non-intersecting Brownian bridges have been studied extensively using the determinantal point process point of view. It is known that as $n\to\infty$, the top path $X_n(t)$ converges to the curve $x= 2\sqrt{nt(1-t)}$, $0\le t\le 1$, and the fluctuations around the curve in an appropriate scaling is given by the Airy process $\mathcal{A}(\tau)$ \cite{Prahofer02}.
Especially near the peak location it is known that
(see e.g. \cite{Johansson04},
\cite{Alder09outlies})
\begin{equation}\label{eq:2}
  2n^{1/6}
  \left(X_{n}\left(\frac{1}{2}+\frac{2\tau}{n^{1/3}}\right)-\sqrt{n}\right)
  \rightarrow \mathcal{A}(\tau)-\tau^2
\end{equation}
in the sense of finite distribution.
By symmetry, $-X_1$ has the same fluctuations.
It is reasonable to expect that the fluctuations of the top path and the bottom path become
independent near $t=\frac12$ as $n\to\infty$.
Therefore, it is natural to conjecture that
\begin{equation}\label{eq:3}
  	2n^{1/6}(\widthf_n-2\sqrt{n})
  	\Rightarrow \max_{\tau\in \R} \left(\mathcal{A}^{(1)}(\tau)+\mathcal{A}^{(2)}(\tau)-2\tau^2 \right)
\end{equation}
where $\mathcal{A}^{(1)}$ and $\mathcal{A}^{(2)}$ are two independent copies of Airy processes.
Combining~\eqref{eq:3} and~\eqref{eq:thm1formula}, we expect the following interesting identity:
\begin{equation}
\label{z11-3-1}
    2^{-1/3} \cdot \max_{\tau\in \R} \left(\mathcal{A}^{(1)}(\tau)+\mathcal{A}^{(2)}(\tau)-2\tau^2 \right)  = \chi_{GUE},
\end{equation}
where $\chi_{GUE}$ is the GUE Tracy-Widom random variable.
Indeed we have the following identity:

\begin{thm}\label{z11-3-4}
Let $\mathcal{A}^{(1)}$ and $\mathcal{A}^{(2)}$ be two independent copies of Airy processes.
Then for any positive constants $\alpha$ and $\beta$,
\begin{equation}
\label{z11-7-1}
    (\alpha+\beta)^{-1/3} \cdot \max_{\tau\in \R} \left(\alpha^{1/3}\mathcal{A}^{(1)}(\alpha^{-2/3}\tau)
    +\beta^{1/3}\mathcal{A}^{(2)}(\beta^{-2/3}\tau)-(\alpha^{-1}+\beta^{-1})\tau^2 \right)  = \chi_{GUE}.
\end{equation}
\end{thm}

It may be possible to establish~\eqref{eq:3} using 
the results obtained in \cite{CorwinHammond}, and therefore prove this theorem using~\eqref{eq:thm1formula}.
However, we do not follow this approach and instead give an independent proof of Theorem~\ref{z11-3-4}.
The proof is obtained by considering the point-to-point directed last passage time of a solvable directed last passage percolation model in two different ways.
This indirect proof is analogous to the proof of Johansson \cite{Johansson04} for the identity
\begin{equation}\label{eq:GOE}
    2^{2/3} \cdot \max_{\tau\in \R} \left(\mathcal{A}(\tau)-\tau^2 \right)  = \chi_{GOE},
\end{equation}
where $\chi_{GOE}$ stands for the GOE Tracy-Widom random variable.
Indeed~\eqref{z11-7-1} follows easily from the estimates already established in \cite{Johansson04}.
The proof is given in Section~\ref{sec:LPP}.
Considering other versions of directed last passage percolation models, one can also obtain other identities involving Airy processes and Brownian motions. See \cite{CorwinLiuWang}.

A  direct proof of~\eqref{eq:GOE} was recently obtained in \cite{CorwinQuastelRemenik}. This paper also obtained a Fredholm determinant formula for $\Prob(\mathcal{A}(\tau)\le g(\tau), \,  t\in [-L,L])$ for general non-random functions $g$. 
It is an interesting question to generalize this approach to random functions $g$ and use it to give a direct proof of~\eqref{z11-7-1}.

\subsubsection*{Organization of paper}

This paper is organized as follows.
In Section~\ref{sec:THf} we prove Theorem~\ref{thm:0} and its Hankel version. The proof of Corollary~\ref{cor:00} is also given in this section. The results on the width of non-intersecting Brownian processes, Proposition~\ref{prop:BMHank} and Theorem~\ref{thm:1}, are presented in Section~\ref{sec:NIB}.
The analogous results for symmetric simple random walks in both continuous-time and discrete-time are in Section~\ref{sec:SRW}.
Finally, the proof of Theorem~\ref{z11-3-4} is given in Section~\ref{sec:LPP}.

\subsubsection*{Acknowledgments}
We would like to thank Ivan Corwin, Gregory Schehr, and Dong Wang for useful conversations.
The work of Jinho Baik was supported in part by NSF grants DMS1068646.

\section{Discrete Toeplitz and Hankel determinants}\label{sec:THf}

In this section we prove Theorem~\ref{thm:0} and its Hankel version. At the end we prove Corollary~\ref{cor:00}.

\begin{proof}[Proof of Theorem~\ref{thm:0}]
From the Cauchy's integral formula,
\begin{equation}
  \sum_{z\in\nodes}z^{-j+k}f(z)=\int_{\Sigma_{out}\cup (-\Sigma_{in})}z^{-j+k}f(z)\frac{z\gamma'(z)}{\gamma(z)}\frac{dz}{2\pi iz}.
\end{equation}
Inserting this into the definition of the discrete Toeplitz determinant and performing simple row and column operations, we find that $\To_n(f, \nodes)$ equals
\begin{equation}
\label{z10-28-1}
  \frac{1}{\kappa_0^2\cdots\kappa_{n-1}^2}\det\left[\int_{\Sigma_{out}\cup (-\Sigma_{in})}\overline{p_j(\bar z^{-1})}p_k(z)f(z)\frac{z\gamma'(z)}{|\nodes|\gamma(z)}\frac{dz}{2\pi iz}\right]_{j,k=0}^{n-1}.
\end{equation}
Note that $\kappa_j$ are positive by definition.
Now from the general theory of orthogonal polynomials, $\frac1{\kappa_0^2\cdots\kappa_{n-1}^2}$  is precisely the continuous Toeplitz determinant $\ToC_n(f)$.
The orthonormality conditions of $p_k$ are $\delta_{jk}= \int_{|z|=1} \overline{p_j(z)}p_k(z) f(z)\frac{dz}{2\pi iz}$.
Using the fact that $z=\bar z^{-1}$ on the circle and using the analyticity of $f$
, these conditions imply that $\delta_{jk} = \int_{\Sigma_{out}} \overline{p_j(\bar z^{-1})}p_k(z)f(z)\frac{dz}{2\pi iz}$.
Using this
the determinant in~\eqref{z10-28-1} can be written as
\begin{equation}
  \det\left[\delta_{jk}+\int_{\Sigma_{out}\cup \Sigma_{in}}\overline{p_j(\bar z^{-1})}p_k(z)f(z)v(z)\frac{dz}{2\pi iz}\right]_{j,k=0}^{n-1}
\end{equation}
with $v$ defined in~\eqref{eq:vformu}.
Now the theorem follows by applying the general identity $\det(1+AB)=\det(1+BA)$ and using the Christoffel-Darboux formula.
\end{proof}

\begin{rmk}\label{thm:3}
Theorem~\ref{thm:0} can be slightly generalized as follows.
Let $b(z)$ be a non-trivial analytic function in a neighborhood of $\Sigma$ such that $b(z)\ge 0$ for $z\in \Sigma$ and let
$p_k$ be the orthonormal polynomials with respect to the measure $f(z)b(z) \frac{dz}{2\pi iz}$ on the circle.
Then
\begin{equation}
  \label{z10-30-3}
  	\To_n(f, \nodes)=\ToC_n(fb)\det\left(1+K\right)_{L^2(\Sigma_{in}\cup \Sigma_{out},\frac{dz}{2\pi iz})}
\end{equation}
with
\begin{equation}
  v(z)=\begin{cases}
    -\frac{z\gamma'(z)}{|\nodes|\gamma(z)}, \quad & z\in\Sigma_{in},\\
    \frac{z\gamma'(z)}{|\nodes|\gamma(z)}-b(z), \quad & z\in\Sigma_{out}.
  \end{cases}
\end{equation}
The proof is essentially  same.
\end{rmk}

The Hankel version is as follows. The proof is almost same as that of Theorem~\ref{thm:0} and we do not present it.
Assume:
\begin{enumerate}[(a)]
\item Let $\nodes$ be a (either finite or infinite) discrete subset of $\R$ with no accumulating points.
\item
Let $f(x)\ge 0$ be a non-trivial function on $\R$ which is analytic in a neighborhood $\Omega=\{z=x+iy: x\in \R, |y|<\delta\}$ of $\R$ for some $\delta>0$. We also assume that the discrete Hankel determinant $\Ha_n(f, \nodes)$ is well defined.
\item
Let $b(z)$ be a non-trivial analytic function in $\Omega$
such that $b(x)\ge 0$ for $x\in \R$ and $|z|^k |f(z)b(z)|\to 0$ as $|z|\to \infty$ in $\Omega$ for every $k\ge 0$.
\item
Let $\gamma(z)$ be an analytic function in $\Omega$ such that
$\gamma(x)$ vanishes exactly on $\nodes$, all the roots are simple, and
$|z^k f(z) \frac{\gamma'(z)}{\gamma(z)}|\to 0$ as $z\in \Omega$, $|Re(z)|\to \infty$ for all $k\ge 0$.
\end{enumerate}

Let $p_k(x)$ be the (continuous) orthonormal  polynomials with respect to the weight $f(x)b(x)$ on $\R$.
Let $\kappa_k$ denote the leading coefficient of $p_k(x)$.
Set the Christoffel-Darboux kernel
\begin{equation}
    	K_{CD}(z,w) =
	\frac{\kappa_{n-1}}{\kappa_n} \frac{p_n(z)p_{n-1}(w)-p_{n-1}(z)
p_n(w)}{z-w}.
\end{equation}
\begin{thm}
\label{thm:2}
Assuming (a)--(d) above, we have
\begin{equation}
  	\Ha_n(f, \nodes) =\HaC_n(fb)\det\left(1+K\right)_{L^2(C_{+}\cup C_{-},dz)},
\end{equation}
where $K$ is the integral operator with kernel
\begin{equation}
    	K(z,w) =
	K_{CD}(z,w) \sqrt{f(z)f(w)v(z)v(w)}
\end{equation}
where $C_{\pm}= \R \pm i\delta/2$, oriented from left to right, and
\begin{equation}\label{eq:vformu2}
  	v(z):=\begin{cases}
    -\frac{\gamma'(z)}{2\pi i\gamma(z)}-\frac{b(z)}{2},&z\in C_+,\\
    \frac{\gamma'(z)}{2\pi i \gamma(z)}-\frac{b(z)}{2},&z\in C_-.
  \end{cases}
\end{equation}
\end{thm}


We now prove Corollary~\ref{cor:00}.

\begin{proof}[Proof of Corollary~\ref{cor:00}]
Let $\epsilon>0$ be small enough so that $f(z)$ is analytic 
in the annulus $1-2\epsilon<|z|<1+2\epsilon$.
We now apply Theorem~\ref{thm:0} where we take $\Sigma_{in}$ and $\Sigma_{out}$  as the circles of radii $1-\epsilon$ and $1+\epsilon$ respectively.
Using the fact that the Fredholm determinant is invariant under conjugations, it is enough to prove that
\begin{equation}
\label{z10-29-1}
	  |(z/w)^{n/2}K(z,w)|=O(e^{-c(m-n)})
\end{equation}
uniformly for $z,w\in \Sigma_{in}\cup\Sigma_{out}$, for some constant $c>0$.

The asymptotics of orthonormal polynomials with respect to a fixed measure of form $f(z)\frac{dz}{2\pi z}$ on the unit circle are well known (see, for example, \cite{Szego}). 
When $f$ is positive and analytic on the circle, an explicit asymptotic expansion of $p_n(z)$  as $n\to \infty$ for all complex $z$ can be found in \cite{MFMS}. These results imply that
for a given $\epsilon>0$, there is a constant $C>0$ such that
\begin{equation}
      p_n(z), \,  p_n'(z)=\begin{cases}
        z^n O(e^{-Cn}),\qquad &|z|\ge 1+\epsilon,\\
        O(e^{-Cn}),&|z|\le 1-\epsilon,
      \end{cases}
\end{equation}
uniformly. Since $p_n^*(z)=z^n\overline{p_n(1/{\bar z})}$, the above estimates also hold with $p_n(z)$ replaced by $p_n^*(z)$.
Inserting these into~\eqref{eq:Kconti}, we find that $|(z/w)^{n/2}K_{conti}(z,w)|$ is
     \begin{equation}\label{eq:Kconzwcn}
       \begin{cases}
                (1-\epsilon)^{-n}(1+O(e^{-2Cn})),\qquad &|z|=|w|=1-\epsilon,\\
         (1+\epsilon)^{n}(1+O(e^{-2Cn})),\qquad &|z|=|w|=1+\epsilon,\\
          \left(\frac{1+\epsilon}{1-\epsilon}\right)^{n/2}(1+O(e^{-2Cn})),& |z|=1\mp \epsilon, \, |w|=1\pm\epsilon.
       \end{cases}
     \end{equation}

On the other hand, from the formula~\eqref{eq:vforCor}, it is easy to check that
     \begin{equation}\label{eq:vcorpfr}
       |v(z)|\le \begin{cases}
                2(1-\epsilon)^{m},\qquad &|z|=1-\epsilon,\\
         2(1+\epsilon)^{-m},\qquad &|z|=1+\epsilon,
       \end{cases}
     \end{equation}
for all large enough $m$. Inserting~\eqref{eq:Kconzwcn} and~\eqref{eq:vcorpfr} into~\eqref{eq:KforToe}, we obtain~\eqref{z10-29-1}. This completes the proof.
\end{proof}

\section{Non-intersecting Brownian bridges}\label{sec:NIB}

\subsection{Hankel determinant formula}\label{sec:HBMfo} 

We prove Proposition~\ref{prop:BMHank}.

Let $D_n:=\{x_0<x_1<\cdots<x_{n-1}\}\subset \R^n$.
Fix $\alpha=(\alpha_0, \cdots, \alpha_{n-1})\in D_n$ and
$\beta=(\beta_0, \cdots, \beta_{n-1})\in D_n$.
Let $X(t)=(X_0(t),X_1(t),\cdots,X_{n-1}(t))$ be  $n$ independent standard Brownian motions.
We denote the conditional probability that $X(0)=\alpha$ and $X(1)=\beta$ by $\Prob_{\alpha, \beta}$.
Let $\Ne_0$ be the event that $X(t)\in D_n$ for all $t\in (0,1)$
and let $\Ne_1$ be the event that $X(t)\in D_n(\width):= \{ x_0<x_1<\cdots<x_{n-1}< x_0+\width\}$.
Then $\Prob (\widthf_n < \width)$ may be computed by taking the limit of
$\frac{\Prob_{\alpha, \beta}( \Ne_1 )}{\Prob_{\alpha, \beta}(\Ne_0)}$ as
$\alpha, \beta\to 0$.

From the Karlin-McGregor argument \cite{KM},
$\Prob_{\alpha, \beta}(\Ne_0)= \frac{\det\left[p(\alpha_j-\beta_k)\right]_{j,k=0}^{n-1}}{\prod_{j=0}^{n-1}p(\alpha_j-\beta_j)}$,
where $p(x)=\frac{1}{\sqrt{2\pi}} e^{-\frac{x^2}{2}}$.
On the other hand, the Karlin-McGregor argument in the chamber $D_n(\width)$ was given for example in \cite{HobsonWerner} and implies the following. For convenience of the reader, we include a proof.

\begin{lm}
\label{z10-31-1} The probability
$\Prob_{\alpha, \beta}(\Ne_1 )$ equals
\begin{equation}
\label{determinant formula for BM1}
	\frac1{\prod_{j=0}^{n-1}p(\alpha_j-\beta_j)} \sum_{\substack{h_j\in \Z \\ h_0+h_1+\cdots+h_{n-1}=0}}
  \det\left[ p(\alpha_j-\beta_k+h_k\width)\right]_{j,k=0}^{n-1}.
\end{equation}
\end{lm}


\begin{proof}
For $\beta=(\beta_0, \cdots, \beta_{n-1})\in D(M)$,
let $\lift_\width(\beta)$ be the set of all $n$-tuples $(\beta'_{0}+h_0\width,\cdots,\beta'_{n-1}+h_{n-1}\width)$ where $(\beta'_0,\cdots,\beta'_{n-1})$ is an re-arrangment of $(\beta_0, \cdots, \beta_{n-1})$ and  $h_0,\cdots,h_{n-1}$ are $n$ integers of which the sum is $0$.  The key property of $\lift_\width(\beta)$ is that $\lift_\width(\beta)\cap D_n(\width)=\{\beta\}$.
Indeed note that since $\beta\in D_n(\width)$, we have $|\beta'_i-\beta'_j|<\width$ for all $i,j$.
Thus if $(\beta'_{0}+h_0\width,\cdots,\beta'_{n-1}+h_{n-1}\width)\in D_n(\width)$, then we have $h_0\le\cdots\le h_{n-1}\le h_0+1$. Since $h_0+\cdots+h_{n-1}=0$, this implies that $h_0=\cdots=h_{n-1}=0$.
This implies that $\beta'_j=\beta_j$ for $j$ and $\lift_\width(\beta)\cap D_n(\width)=\{\beta\}$.

Now we consider $n$ independent standard Brownian motions $X(t)$, $0\le t\le 1$, satisfying $X(0)=\alpha$ and $X(1)\in\lift_\width(\beta)$. Then one of the following two events happens:

(a) $X(t)\in D_n(\width)$ for all $t\in [0,1]$. In this case, $X(1)=\beta$.

(b) There exists a smallest time $t_{min}$ such that $X(t_{min})$ is on the boundary of the chamber $D_n(\width)$. Then almost surely one of the following two events happens: (b1) a unique pair of two neighboring Brownian motions intersect each other at time $t_{min}$,  (b2) $X_{n-1}(t_{min})-X_0(t_{min})=\width$.
By exchanging the two corresponding Brownian motions after time $t_{min}$ in the case (b1), or replacing $X_0(t),X_{n-1}(t)$ by $X_{n-1}(t)-\width,X_0(t)+\width$ respectively after time $t_{min}$ in the case (b2), we obtain two new Brownian motions.
Define $X^*(t)$ be the these two new Brownian motions together with the other $n-2$ Brownian motions.
Then clearly, $X^*(1)\in \lift_\width(\beta)$. It is easy to see that 
$(X^*)^*(t)=X(t)$ and hence this defines an involution on the event (b) almost surely.
By expanding the determinant in the sum in~\eqref{determinant formula for BM1} and applying the involution, we find that that this sum equals the probability that $X(t)$ is from $\alpha$ to $\beta$ such that $X(t)$ stays in $D_n(M)$. Hence Lemma~\ref{z10-31-1} follows.
\end{proof}

Define the  generating function
\begin{equation}\label{z10-04-2}
  g(x,\theta):=\sum_{h\in\Z}p(x+h\width)e^{i \width
  h\theta}.
\end{equation}
It is direct to check that the sum in ~\eqref{determinant formula for BM1} equals
$\frac{\width}{2\pi}\int_0^{\frac{2\pi}{\width}}
\det\left[ g(\alpha_j-\beta_k,\theta)\right]_{j,k=0}^{n-1} d\theta$.	
Thus, we find that
\begin{equation}\label{z9-28-1}
  	\frac{\Prob_{\alpha, \beta}( \Ne_1 )}{\Prob_{\alpha, \beta}(\Ne_0)}
  	=\frac{\frac{\width}{2\pi}\int_0^{\frac{2\pi}{\width}}
  \det\left[ g(\alpha_j-\beta_k,\theta)\right]_{j,k=0}^{n-1}d\theta}
  {\det\left[p(\alpha_j-\beta_k)\right]_{j,k=0}^{n-1}}.
\end{equation}
By taking the limit $\alpha,\beta\to 0$, we obtain:

\begin{lm}
\label{Hankel determinant formula for BM}
We have
 \begin{equation}
  \label{z9-29-3}
    \Prob\left(\widthf_n< \width\right)
    =\int_0^1 \frac{\left(\frac{\sqrt{2}\pi}{\width\sqrt{n}}\right)^n \sum_{x\in \nodes_{\s}^n}\Delta(x)^2\prod_{j=0}^{n-1} e^{-nx_j^2}}
    {\int_{x\in\R^n}\Delta(x)^2\prod_{j=0}^{n-1} e^{-nx_j^2}dx_j}d\s,
  \end{equation}
  where $\nodes_{\s}:=\left\{\frac{\sqrt{2}\pi}{\width\sqrt{n}}(m-{\s}): m\in\Z\right\}\subset \R$
    and $\Delta(x)$ denotes the the Vandermonde determinant of $x=(x_0, \cdots, x_{n-1})$.
\end{lm}

\begin{proof}
We insert $p(x)=\frac{1}{\sqrt{2\pi}} e^{-\frac{x^2}{2}}$ into~\eqref{z10-04-2} and then
use the Poisson summation formula to obtain
\begin{equation}\label{z10-04-3}
  	g(x,\theta)
	=\frac{1}{\width} \sum_{h\in\Z}
e^{-\frac{1}{2}(\frac{2\pi h}{\width}-\theta)^2
	+ i x(\frac{2\pi h}{\width}-\theta)}.
\end{equation}
Using the  Andreief's formula \cite{Andreief}, $\det\left[ g(\alpha_j-\beta_k,\theta)\right]_{j,k=0}^{n-1}$ equals
\begin{equation}
\begin{split}
     \frac{1}{n!\width^n}\sum_{h\in\Z^n} \det\left[ e^{i\alpha_j(\frac{2\pi h_k}M-\theta) } \right]_{j,k=0}^{n-1}
     \det\left[ e^{-i\beta_j(\frac{2\pi h_k}M-\theta) }  \right]_{j,k=0}^{n-1}
    \prod_{j=0}^{n-1}
    e^{-\frac{1}{2}(\frac{2\pi h_j}{\width}-\theta)^2}.
\end{split}
\end{equation}
Since $\det\left[e^{x_jy_k}\right]_{j,k=0}^{n-1}=c\Delta(x)\Delta(y)(1+O(y))$
with $c=\prod_{j=0}^{n-1}\frac{1}{j!}$ as $y\to 0$ for each $x$,
we find that
\begin{equation}\label{z10-04-5}
  \lim_{\alpha,\beta\to 0}\frac{\det\left[ g(\alpha_j-\beta_k,\theta)\right]_{j,k=0}^{n-1} }{c^2\Delta(\alpha)\Delta(\beta)}
  =\frac{(2\pi/\width)^{n(n-1)}}{n!\width^n}\sum_{h\in\Z^n}\Delta\left(h\right)^2  \prod_{j=0}^{n-1}
   e^{-\frac{1}{2}(\frac{2\pi h_j}{\width}-\theta)^2}
\end{equation}

On the other hand, using $p(x)= \frac{1}{2\pi}\int_{\R} e^{-\frac{1}{2}y^2+ixy}dy$,
\begin{equation}\label{z10-04-6}
  \lim_{\alpha,\beta\to 0}\frac{\det\left[ p(\alpha_j-\beta_k)\right]_{j,k=0}^{n-1}}{c^2\Delta(\alpha)\Delta(\beta)}
  =\frac{1}{(2\pi)^nn!}\int_{h\in\R^n}\Delta(h)^2
\prod_{j=0}^{n-1} e^{-\frac{1}{2}h_j^2} dh_j.
\end{equation}
Inserting~\eqref{z10-04-5} and~\eqref{z10-04-6} into~\eqref{z9-28-1}, we obtain~\eqref{z9-29-3} after appropriate changes of variables.
\end{proof}

Proposition~\ref{prop:BMHank} follows from Lemma~\ref{Hankel determinant formula for BM} immediately.

\subsection{Proof of Theorem~\ref{thm:1}}\label{sec:pfTh12}

We apply Theorem~\ref{thm:2} to Proposition~\ref{prop:BMHank}.
Set
\begin{equation}
	\al = \al_{M,n} := \frac{\width\sqrt{n}}{\sqrt{2}\pi}.
\end{equation}
Noting that $\al^{-n} \Ha_n(F, \nodes_s)= \Ha_n(\al^{-1} F, \nodes_s)$, we set
\begin{equation}
	f(z)=\al^{-1}e^{-nz^2}, \quad b(z)=\al,
	\quad \gamma(z)=\sin\left(\pi(\al z+s)\right)
\end{equation}
in Theorem~\ref{thm:2}.
Then $v(z)=v_\s(z)d$, where
\begin{equation}\label{eq:vsd}
  	v_\s(z):=\begin{cases}
             -\frac{\cos\left(\pi(\al z+s)\right)}
             {2i \sin\left(\pi(\al z+s)\right)}-\frac{1}{2}= \frac{e^{2i\alpha(z)}}{1-e^{2i\alpha(z)}} ,\quad &z\in C_+, \\ 
             \frac{\cos\left(\pi(\al z+s) \right)}
             {2i \sin\left(\pi(\al z+s) \right)}-\frac{1}{2}= \frac{e^{-2i\alpha(z)}}{1-e^{-2i\alpha(z)}} , \quad &z\in C_-,  
           \end{cases}
\end{equation}
where $\alpha(z)= \pi(\al z+s)$.
Let $p_j(x)=\kappa_j x^j+\cdots$ be the orthonormal polynomials with respect to
$f(x)b(x)=e^{-nx^2}$ on $\R$
and set
\begin{equation}
    	K_{CD}(z,w) =
	\frac{\kappa_{n-1}}{\kappa_n} \frac{p_n(z)p_{n-1}(w)-p_{n-1}(z)
p_{n}(w)}{z-w}.
\end{equation}
Then from Theorem~\ref{thm:2},
\begin{equation}\label{eq:asBMP}
    \Prob\left(\widthf_n<  \width\right)=\int_0^1 P_\s(\width)d\s,
    \quad P_\s(\width)=\det\left(1+K_\s\right)_{L^2(C_+\cup C_-,dz)}.
\end{equation}
where
\begin{equation}
    K_\s(z,w)= K_{CD}(z,w)
	v_\s(z)^{\frac{1}{2}}v_\s(w)^{\frac{1}{2}}e^{-\frac{n}{2}(z^2+w^2)}.
\end{equation}
We set (see~\eqref{eq:thm1formula})
\begin{equation}\label{eq:wsca0}
  	\width=2\sqrt{n}+2^{-2/3}n^{-1/6}\x,
\end{equation}
where $\x\in\R$ is fixed.

The asymptotic of $P_\s(\width)$ is obtained in two steps.
The first step is to find the asymptotics of the orthonormal polynomials for $z$ in complex plane.
The second step is to insert them into the formula of $K_\s$ and then to prove the convergence of an appropriately scaled operator in trace class.
It turns out that the most important information is the asymptotics of the orthonormal polynomials for $z$ close to $z=0$ with order $n^{-1/3}$.
Such asymptotics can be obtained from the method of steepest-descent applied to the integral representation of Hermite polynomials.
However, here we proceed using the Riemann-Hilbert method as a way of illustration since the orthonormal polynomials for the other two non-intersecting processes to be discussed in the next section are not classical and hence lack the integral representation.

For the weight $e^{-nx^2}$, the details of the asymptotic analysis of the Riemann-Hilbert problem can be found in \cite{DKMVZ3} and \cite{DeiftBook}. 
Let $Y(z)$ be the (unique) $2\times 2$ matrix which (a) is analytic in $\C\backslash\R$,
(b) satisfies $Y_+(z)=Y_-(z)\big(\begin{smallmatrix}1&e^{-nz^2}\\0&1
    \end{smallmatrix}\big)$ for $z\in\R$,
    and (c) $Y(z)=(1+O(z^{-1}))\big(\begin{smallmatrix}z^n&0 \\ 0 & z^{-n}
    \end{smallmatrix}\big)$ as $z\to\infty$.
It is well-known (\cite{FIK}) that
\begin{equation}
\label{z9-30-1}
  K_{CD}(z,w)=
	\frac{Y_{11}(z)Y_{21}(w)-Y_{21}(z)Y_{11}(w)}{-2\pi\ii(z-w)}.
\end{equation}
Let
\begin{equation}
  	g(z):=\frac{1}{\pi}\int_{-\sqrt{2}}^{\sqrt{2}}\log(z-s)\sqrt{2-s^2}ds
\end{equation}
be the so-called $g$-function. Here  $\log$ denotes the
the principal branch of the logarithm.
It can be checked that $-g_+(z)-g_-(z)+z^2$ is a constant independent of $z\in (-\sqrt{2},\sqrt{2})$. Set $\ell$ to be this constant:
\begin{equation}
  	l:=-g_+(z)-g_-(z)+z^2, \quad z\in(-\sqrt{2},\sqrt{2}).
\end{equation}
Set
\begin{equation}\label{eq:betade}
  m_{\infty}(z):=
  \begin{pmatrix}
                      \frac{\beta+\beta^{-1}}{2}&\frac{\beta-\beta^{-1}}{2\ii}\\
                      \frac{\beta-\beta^{-1}}{-2\ii}&\frac{\beta+\beta^{-1}}{2}
     \end{pmatrix},
                \qquad \beta(z):=\left(\frac{z-\sqrt{2}}{z+\sqrt{2}}\right)^{1/4},
\end{equation}
where the function $\beta(z)$ is defined to be analytic in $\C\backslash[-\sqrt{2},\sqrt{2}]$ and to satisfy
$\beta(z)\rightarrow1$ as $z\rightarrow\infty$.
Then the asymptotic results from the Riemann-Hilbert analysis is given in Theorem 7.171 in \cite{DeiftBook}:
\begin{equation}
\label{z9-29-7}
  Y(z)=e^{-\frac{nl}{2}\sigma_3}(\idn+Er(n,z))m_{\infty}(z)e^{\frac{nl}{2}\sigma_3}e^{ng(z)\sigma_3},
  \quad z\in\C\backslash\R,
\end{equation}
where the error term $Er(n,z)$ satisfies (see the remark after theorem 7.171)
$\sup_{|{\mathrm Im}\,z|\ge \eta} |Er(n,z)|\le \frac{C(\eta)}{n}$
for a positive constant $C(\eta)$, for each $\eta>0$.
An inspection of the proof shows that the same analysis yields the following estimate. The proof is basically the same and we do not repeat.

\begin{lm}
\label{z11-29-3}
Let $\eta>0$. There exists a constant $C(\eta)>0$ such that for each $0<\alpha<1$,
  \begin{equation}
     \sup_{z\in D_n} |Er(n,z)|\le \frac{C(\eta)}{n^{1-\alpha}},
  \end{equation}
where $D_n:=\{z: |{\mathrm Im}\,z|> \frac{\eta}{n^{\alpha}},|z\pm \sqrt{2}|>\eta\}$.
\end{lm}


We now insert~\eqref{z9-29-7} into~\eqref{z9-30-1}, and  find the asymptotics of $K$.
Before we do so, we first note that the contours $C_+$ and $C_-$ in the formula of $P_\s(\width)$
can be deformed thanks to the Cauchy's theorem.
We choose the contours as follows, and we call them $C_1$ and $C_2$ respectively. Let $C_1$ be an infinite simple contour in the upper half-plane of shape shown in Figure~\ref{fig:BM}
satisfying
\begin{equation}
\label{z9-29-6}
  \text{dist}(\R,C_1)=O(n^{-1/3}),\qquad \text{dist}(\pm\sqrt{2},C_1)=O(1).
\end{equation}
Set $C_2=\overline{C_1}$. Later we will make a more specific choice of the contours.
\begin{figure}[htbp]
  \centering
  \includegraphics[scale=0.65]{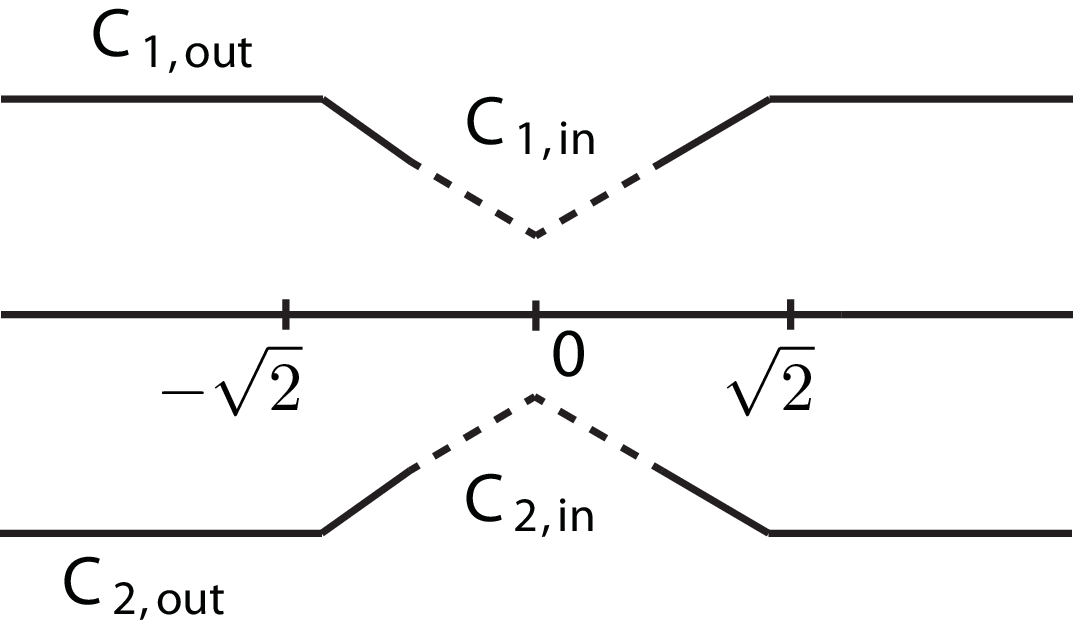}
\caption{$C_1=C_{1,\oou}\cup C_{1,\iin}, \quad C_{2}=C_{2,\oou}\cup C_{2,\iin}$}\label{fig:BM}
\end{figure}
Then from Lemma~\ref{z11-29-3}, $Er(n,z)= O(n^{-2/3})$ for $z\in C_1\cup C_2$.
Also since $\beta(z)=O(1)$, $\beta(z)^{-1}=O(1)$,
and $\text{arg}(\beta(z))\in \left(-\frac{\pi}{4},\frac{\pi}{4}\right)$ for $z\in C_1\cup C_2$,
we have $\frac{\beta-\beta^{-1}}{\beta+\beta^{-1}}=O(1)$ for $z\in C_1\cup C_2$.
Thus, we find from~\eqref{z9-29-7} that
\begin{equation}
\label{z9-29-8}
    Y_{11}(z)=e^{ng(z)}\frac{\beta(z)+\beta(z)^{-1}}{2}(1+O(n^{-2/3}))
\end{equation}
and
\begin{equation}
\label{z9-29-9}
  Y_{21}(z)=e^{ng(z)+nl}\left(O(n^{-2/3})+\frac{\beta(z)-\beta(z)^{-1}}{-2\ii}(1+O(n^{-2/3}))\right)
\end{equation}
for $z\in C_1\cup C_2$.
On the other hand, from the definition~\eqref{eq:vsd} of $v_\s$ and the choice of $C_1$
there exists a positive constant $c$ such that
\begin{equation}
  \begin{split}
    v_\s(z)&=
    \begin{cases}
    e^{ 2\ii\alpha(z)} (1+O(e^{-c n^{1/6}})), \quad z\in C_1, \\
    e^{ -2\ii\alpha(z)} (1+O(e^{-c n^{1/6}})), \quad z\in C_2,
    \end{cases}
  \end{split}
\end{equation}
where $\alpha(z)= \pi(\al z+s)= \frac{\width\sqrt{n}}{\sqrt{2}}z+\s\pi$ is defined earlier.
Therefore, we find that for $z, w\in C_1\cup C_2$,
\begin{equation}
\label{z12-13-1}
  	K_\s(z,w)
	=\frac{f_1(z)f_2(w)-f_2(z)f_1(w)}{-2\pi\ii(z-w)} e^{n\phi(z)+n\phi(w)},
\end{equation}
where
\begin{equation}\label{eq:phid}
  \phi(z):=
  \begin{cases}
  g(z)-\frac{1}{2}z^2+\frac{1}{2}l +\frac{\ii\width }{\sqrt{2n}}z,\qquad& \textrm{Im} (z)>0,\\ 
  g(z)-\frac{1}{2}z^2+\frac{1}{2}l -\frac{\ii\width }{\sqrt{2n}}z,\qquad& \textrm{Im} (z)<0, 
  \end{cases}
\end{equation}
and $f_1,f_2$ are both analytic in $\C\backslash\R$ and satisfy 
\begin{equation}\label{eq:f1f}
  f_1(z)=\begin{cases}
           e^{\ii\s\pi} \frac{\beta(z)+\beta(z)^{-1}}{2}(1+O(n^{-2/3})),\qquad& z\in C_1,\\
           e^{-\ii\s\pi} \frac{\beta(z)+\beta(z)^{-1}}{2}(1+O(n^{-2/3})),\qquad& z\in C_2,
         \end{cases}
\end{equation}
\begin{equation}\label{eq:f2f}
  f_2(z)=\begin{cases}
           e^{\ii\s\pi} \left(O(n^{-2/3})+\frac{\beta(z)-\beta(z)^{-1}}{-2\ii}(1+O(n^{-2/3}))\right),\qquad& z\in C_1,\\
           e^{-\ii\s\pi} \left(O(n^{-2/3})+\frac{\beta(z)-\beta(z)^{-1}}{-2\ii}(1+O(n^{-2/3}))\right),\qquad& z\in C_2.
         \end{cases}
\end{equation}
Note that $f_1(z)$, $f_2(z)$, and their derivatives are bounded on $C_1\cup C_2$.

So far we only used the fact that the contours $C_1$ and $C_2$ satisfy the conditions~\eqref{z9-29-6}. Now we make a more specific choice of the contours as follows (see Figure~\ref{fig:BM}).
For a small fixed $\epsilon>0$ to be chosen in Lemma~\ref{z12-7-1}, set
\begin{equation}\label{eq:sigmaa}
   \Sigma =\{u+\ii v : \, -\epsilon \le u \le \epsilon, \, v=n^{-1/3}+|u|/\sqrt{3} \}.
\end{equation}
Define $C_{1,in}$ to be the part of $\Sigma$ such that $|u|\le n^{-1/4}$:
\begin{equation}
  C_{1,\iin}=\{u+\ii v : \, -n^{-1/4}\le u\le n^{-1/4}, \, v=n^{-1/3}+|u|/\sqrt{3} \}.
\end{equation}
Define
$C_{1, \oou}$ be the union of $\Sigma\setminus C_{1, \iin}$ and the horizontal line segments $u+iv_0$, $|u|\ge \epsilon$ where $v_0$ is the maximal imaginary value of $\Sigma$ given by $v_0= n^{-1/3} + \epsilon/\sqrt{2}$.
Set $C_1= C_{1, \iin}\cup C_{1, \oou}$. Define $C_2=\overline{C_1}$.
It is clear from the definition that the contours satisfy the conditions~\eqref{z9-29-6}.

Recall that (see~\eqref{eq:wsca0}) $\width=2\sqrt{n}+2^{-2/3}n^{-1/6}\x$
where $\x\in\R$ is fixed.
We have

\begin{lm}\label{z12-7-1}
There exist $\epsilon>0$, $n_0\in \mathbb{N}$, and positive constants $c_1$ and $c_2$ such that
with the definition~\eqref{eq:sigmaa} of $\Sigma$ with this $\epsilon$, $\phi(z)$ defined in~\eqref{eq:phid} satisfies
\begin{equation}
  \begin{split}
   \textrm{Re } \phi(z)\le c_1n^{-1}, \qquad& z\in C_{1,\iin}\cup C_{2,\iin},\\
   \textrm{Re } \phi(z)\le -c_2n^{-3/4}, \qquad&z\in C_{1,\oou}\cup C_{2,\oou},
   \end{split}
\end{equation}
for all $n\ge n_0$.
\end{lm}

\begin{proof}  
From the properties of $g(z)$ and $l$,
it is easy to show that
$g(z)-\frac12 z^2+\frac12 l=\int_{z}^{\sqrt{2}}\sqrt{s^2-2}ds$ for $z\in\C\setminus (-\infty, \sqrt{2}]$ (see e.g. (7.60) \cite{DeiftBook}).
Thus,
\begin{equation}\label{eq:phiint}
  	\phi(z) =\int_{z}^{\sqrt{2}}\sqrt{s^2-2}ds \pm  \frac{\ii\width }{\sqrt{2n}}z, \qquad z\in\C_{\pm}.
\end{equation}
This implies that for $\phi_{\pm}(u)$ is purely imaginary for $z=u\in (-\sqrt{2}, \sqrt{2})$ where $\phi_{\pm}$ denotes the boundary values from $\C_{\pm}$ respectively.
Hence for $u\in (-\sqrt{2}, \sqrt{2})$ and $v>0$, $\textrm{Re}\,\phi(u+iv)= \textrm{Re}\,(\phi(u+iv)-  \phi_+(u))$.
For $u^2+v^2$ small enough and $v>0$,
using the Taylor's series about $s=0$ and also~\eqref{eq:wsca0}, we have
\begin{equation}\label{eq:rephif}
  \begin{split}
  \textrm{Re}\,\phi(u+iv)
  &=-\textrm{Re}\left(\int_{u}^{u+\ii v}\sqrt{s^2-2}ds\right)-\frac{\width v}{\sqrt{2n}}\\
  &=-\frac{1}{2^{3/2}}\textrm{Im}\left(\int_{u}^{u+\ii v} (s^2 + O(s^4) )ds\right)-\frac{\x }{2^{7/6}n^{2/3}} v.
 \end{split}
\end{equation}
The integral involving $O(s^4)$ is $O(|u^2+v^2|^{5/2})$.
On the other hand,
\begin{equation}
\label{z3-20-1}
  -\frac{1}{2^{3/2}}\textrm{Im}\left(\int_{u}^{u+\ii v}s^2 ds\right)-\frac{\x v}{2^{7/6}n^{2/3}}
  =-\frac{1}{2^{2/3}3}(3u^2v-v^3)-\frac{\x v}{2^{7/6}n^{2/3}}.
\end{equation}
For $z=u+iv$ such that $v= n^{-1/3} + |u|/\sqrt{3}$ (see~\eqref{eq:sigmaa}),~\eqref{z3-20-1} equals
\begin{equation}
\label{z3-20-2}
  n^{-1}\left(-\frac{2^{7/3}}{3^{5/2}}t^3-\frac{2^{1/3}}{3}t^2+\frac{(2^{1/2}-\x)}{2^{7/6}3^{1/2}}t
  +\frac{1}{2^{7/6}3}(2^{1/2}-3\x)\right),
\end{equation}
by setting $t=|u|n^{1/3}$. 
The polynomial in $t$ is cubic
and is of form $f(t)=-a_1t^3-a_2 t^2+a_3t+a_4$ where $a_1, a_2>0$ and $a_3, a_4\in \R$. 
It is easy to check that this function is concave down for positive $t$.
Hence
\begin{enumerate}[(i)]
\item $\sup_{t\ge 0} f(t)$ is bounded above and
\item there are $c>0$ and $t_0>0$ such that $f(t)\le -ct^3$ for $t>t_0$.
\end{enumerate}
Note that for $z\in C_{1, \iin}$, $t\in [0, n^{1/12}]$. Using $(i)$, we find that~\eqref{z3-20-2} is bounded above by a constant time $n^{-1}$ for uniformly in $z\in C_{1, \iin}$.
Since the integral involving $O(s^4)$ in~\eqref{eq:rephif} is $O(n^{-5/4})$ when $z\in C_{1, \iin}$, we find that
there is a constant $c_1>0$ such that $\textrm{Re}\, \phi(z)\le c_1n^{-1}$ for $z\in C_{1, \iin}$.

Now, for $z=u+iv$ such that $v= n^{-1/3} + |u|/\sqrt{3}$ and $|u|\ge n^{-1/4}$, we have $t=|u|n^{1/3}\ge n^{1/12}$ and hence from (ii),
~\eqref{z3-20-2} is bounded above by $-ct^3n^{-1}=-c|u|^3$ for all large enough $n$.
On the other hand, for such $z$, the integral involving $O(s^4)$ in~\eqref{eq:rephif} is $O(|z|^5)= O(|u|^5)$. Hence $\textrm{Re}\, \phi(z)\le -c|u|^3 + O(|u|^5)$ for such $z$. Now if we take $\epsilon>0$ small enough, then there is $c_2>0$ such that $\textrm{Re}\, \phi(z)\le -c_2|u|^3$ for $|u|\le \epsilon$.
Combining this, we find that there exist $\epsilon>0$, $n_0\in \mathbb{N}$, and $c_2>0$ such that for $\Sigma$ with this $\epsilon$,
we have $\textrm{Re}\, \phi(z)\le -c_2|u|^3$ for $z=u+iv\in \Sigma\setminus C_{1, \iin}$. Since $|u|\ge n^{-1/4}$ for such $z$, we find  $\textrm{Re}\, \phi(z)\le -c_2 n^{-3/4}$ for  $z\in \Sigma\setminus C_{1, \iin}$.

We now fix $\epsilon$ as above and consider the horizontal part of $C_{1, \oou}$. Note that from~\eqref{eq:phiint}, for fixed $v_0>0$,
\begin{equation}
\begin{split}
   \frac{\partial}{ \partial u} \text{Re}\, \phi(u+iv_0) = \text{Re}\, \phi'(u+iv_0)
   = - \text{Re}\, \sqrt{(u+iv_0)^2-2}.
\end{split}
\end{equation}
It is straightforward to check that this is $<0$ for $u>0$ and $>0$ for $u<0$. Hence the value of $\text{Re}\, \phi(z)$ for $z$ on the horizontal part of $C_{1, \oou}$ is the largest at the end which are the intersection points of the horizontal segments and $\Sigma$. Since $\textrm{Re}\, \phi(z)\le -c_2 n^{-3/4}$ for  $z\in \Sigma\setminus C_{1, \iin}$, we find that the same bound holds for all $z$ on the horizontal segments of $C_{1, \oou}$. Therefore, we obtain $\textrm{Re } \phi(z)\le -c_2n^{-3/4}$ for all $z\in C_{1,\oou}$.

The estimates on $C_2$ follows from the estimates on $C_1$ due to the symmetry of $\phi$ about the real axis.
\end{proof}

Inserting the estimates in Lemma~\ref{z12-7-1} to the formula~\eqref{z12-13-1} and using the fact that
$f_j(z)$, $j=1,2$, and their derivatives are bounded on $C_1\cup C_2$ (see~\eqref{eq:f1f} and~\eqref{eq:f2f}), we find that
\begin{equation}\label{eq:oute}
	K_\s(z,w)\leq O(e^{-c_2n^{1/4}}), \qquad \text{if one of $z$ or $w$ is in $C_{1,\oou}\cup C_{2,\oou}$.}
\end{equation}

We now analyze the kernel $K_\s(z,w)$ when $z, w\in C_{1,\iin}\cup C_{2,\iin}$. We first scale the kernel.
Set
\begin{equation}\label{eq:hatK}
  \hat K_s(\xi,\eta):=2\pi i \cdot i2^{1/6} n^{-1/3}K_\s(\ii 2^{1/6}n^{-1/3}\xi,\ii 2^{1/6}n^{-1/3}\eta).
\end{equation}
We also set
\begin{equation}
\begin{split}
  \Sigma_1^{(n)}:=\left\{u+\ii v: \,  u=2^{-1/6}+ 3^{-1/2}|v|, \, -2^{-1/6}n^{1/12}\le v\le 2^{-1/6}n^{1/12}\right\}.
\end{split}
\end{equation}
This contour is oriented from top to bottom.
Note that if $\zeta\in \Sigma_1^{(n)}$, then
\begin{equation}
\label{z12-13-2}
  	z=\ii 2^{1/6}n^{-1/3}\zeta\in C_{1,\iin}.
\end{equation}
We also set
$\Sigma_2^{(n)}=\{-\xi:\xi\in\Sigma_1^{(n)}\}$ with the orientation from top to bottom.
Then
\begin{equation}\label{eq:Kssca}
	\det(1+K_s)_{L^2(C_{1,in}\cup C_{2, in}, dz)}= \det(1+\hat K_s)_{L^2(\Sigma_1^{(n)}\cup \Sigma_2^{(n)}, \frac{d\zeta}{2\pi\ii})}.
\end{equation}
From~\eqref{eq:phiint},  
\begin{equation}
  \phi(z)=\begin{cases}
    \frac{\pi\ii}{2}+\left(\frac{\width-2\sqrt{n}}{\sqrt{2n}}\right)\ii z+
  2^{-3/2}3^{-1}\ii z^3+O(z^5),\quad & z\in C_{1,\iin},\\
  -\frac{\pi\ii}{2}-\left(\frac{\width-2\sqrt{n}}{\sqrt{2n}}\right)\ii z-
  2^{-3/2}3^{-1}\ii z^3+O(z^5),\quad & z\in C_{2,\iin}.
  \end{cases}
\end{equation}
This implies that, using~\eqref{eq:wsca0} and $|z|=O(n^{-1/4})$ for $z\in C_{1, \iin}\cup C_{2, \iin}$,
\begin{equation}\label{eq:phize}
  n\phi(\ii 2^{1/6}n^{-1/3}\zeta)=
  \begin{cases}
    \frac{n\pi\ii}{2}+m_\x(\zeta)+O(n^{-1/4}),\qquad & \zeta \in\Sigma_1^{(n)},\\
    -\frac{n\pi\ii}{2}-m_\x(\zeta)+O(n^{-1/4}),\qquad & \zeta \in\Sigma_2^{(n)},
  \end{cases}
\end{equation}
where
\begin{equation}
\label{z3-26-1}
  m_\x(\zeta):=-\frac{1}{2}\x\zeta+\frac{1}{6}\zeta^3, \qquad \zeta\in \C.
\end{equation}
It is also easy to check from the definition~\eqref{eq:betade} that  
\begin{equation}\label{eq:betaas}
  \beta(\ii 2^{\frac{1}{6}}n^{-\frac{1}{3}}\zeta)=
  \begin{cases}
    e^{\frac{\ii\pi}{4}}\left(1-\ii 2^{-\frac{4}{3}} n^{-\frac{1}{3}} \zeta+O(n^{-\frac{1}{2}})\right),\qquad &\zeta\in\Sigma_{1}^{(n)},\\
    e^{\frac{-\ii\pi}{4}}\left(1-\ii 2^{-\frac{4}{3}} n^{-\frac{1}{3}}\zeta +O(n^{-\frac{1}{2}})\right),\qquad &\zeta\in\Sigma_{2}^{(n)}. \end{cases}
\end{equation}
Using these we now evaluate~\eqref{eq:hatK}. Set
\begin{equation}
	z=\ii 2^{1/6}n^{-1/3}\xi, \qquad w=\ii 2^{1/6}n^{-1/3}\eta. 
\end{equation}
We consider two cases separately:
(a) $z,w\in C_{1,\iin}$ or $z,w\in C_{2,\iin}$,
and (b) $z\in C_{1,\iin}, w\in C_{2,\iin}$, or $z\in C_{2,\iin}, w\in C_{1,\iin}$.
From~\eqref{eq:betaas}, 
\begin{equation}\label{eq:casevad}
\begin{split}
	\frac{\beta(z)-\beta(w)}{z-w}= O(1)\qquad \text{for case (a),}
\end{split}
\end{equation}
and
\begin{equation}\label{eq:casevbediff}
\begin{split}
 	\frac{\beta(z)-\beta(w)}{z-w}= \pm n^{1/3} \frac{2^{5/6} \sin\frac{\pi}{4}}{\xi-\eta}(1+O(n^{-\frac{1}{4}}))
	\quad \text{for case (b).}
\end{split}
\end{equation}
Here the sign is $+$ when $z\in C_{1,\iin}, w\in C_{2,\iin}$ and $-$ when $z\in C_{2,\iin}, w\in C_{1,\iin}$.
We also note that using~\eqref{eq:betaas}, for $z\in C_{1, \iin}\cup C_{2, \iin}$
the asymptotic formula~\eqref{eq:f2f} 
can be expressed as
\begin{equation}\label{eq:f2f2}
  f_2(z)=\begin{cases}
           e^{\ii\s\pi} \frac{\beta(z)-\beta(z)^{-1}}{-2\ii} \left(1+ O(n^{-5/12}) \right),\qquad& z\in C_{1,\iin},\\
           e^{-\ii\s\pi} \frac{\beta(z)-\beta(z)^{-1}}{-2\ii} \left(1+ O(n^{-5/12}) \right),\qquad& z\in C_{2, \iin}.
         \end{cases}
\end{equation}
Thus,~\eqref{eq:f1f},~\eqref{eq:betaas}, and~\eqref{eq:casevbediff}, implies that
for case (b),
\begin{equation}
\begin{split}
 \frac{f_1(z)f_2(w)-f_1(z)f_2(w)}{-2\pi\ii(z-w)}
=&-(\beta(z)^{-1}+\beta(w)^{-1})\frac{\beta(z)-\beta(w)}{4\pi(z-w)} \left(1+ O(n^{-5/12}) \right)\\
=&\mp n^{1/3} \frac{\cos(\frac{\pi}{4})\sin(\frac{\pi}{4})}{2^{1/6} \pi (\xi-\eta)}(1+O(n^{-\frac{1}{4}})).
\end{split}
\end{equation}
Inserting this and~\eqref{eq:phize} into~\eqref{z12-13-1} (recall~\eqref{eq:hatK}), we find that
\begin{equation}\label{eq:leadK}
  \hat K_s(\xi,\eta)=\pm\frac{e^{\pm(m_\x(\xi)-m_\x(\eta))}}{\xi-\eta}(1+O(n^{-1/4})),
\end{equation}
for case (b). A similar calculation using~\eqref{eq:casevad} instead of~\eqref{eq:casevbediff} implies that $\hat K_s(\xi,\eta)= O(n^{-1/3})$ for case (a).

The above calculations imply that 
$\hat K_s$ converges to the operator given by the leading term in~\eqref{eq:leadK} or $0$ depending on whether $\xi$ and $\eta$ are on different limiting contours or on the same limiting contours. From this structure, we find that $\hat K_s$ converges to
$\big( \begin{smallmatrix} 0 & K_{12}^{(\infty)} \\  K_{21}^{(\infty)} & 0\end{smallmatrix}\big)$ on $L^2(\Sigma^{(\infty)}_1, \frac{d\zeta}{2\pi i})\oplus L^2(\Sigma^{(\infty)}_2, \frac{d\zeta}{2\pi i})$ in the sense of pointwise limit of the kernel
where
\begin{equation}\label{eq:leadK}
  K_{12}^{(\infty)}(\xi, \eta)=\frac{e^{m_\x(\xi)-m_\x(\eta)}}{\xi-\eta},
  \quad
  K_{21}^{(\infty)}(\xi, \eta)=-\frac{e^{-(m_\x(\xi)-m_\x(\eta))}}{\xi-\eta},
\end{equation}
and  $\Sigma^{(\infty)}_1$ is a simple contour from $e^{\ii\pi/3}\infty$ to $e^{-\ii\pi/3}\infty$ staying in the right half plane, and $\Sigma^{(\infty)}_2=-\Sigma^{(\infty)}_1$ from $e^{2\pi\ii/3}\infty$ to $e^{-2\pi\ii/3}\infty$.
Note that the limiting kernel does not depend on $s$.

In order to ensure that the Fredholm determinant also converges to the Fredholm determinant of the limiting operator, we need additional estimates for the derivatives to establish the convergence in trace norm.
It is not difficult to check that the formal derivatives of the limiting operators indeed yields the correct limits of the derivatives of the kernel.
We do not provide the details of these estimates since the arguments are similar and the calculation follows the standard argument. Then we obtain
\begin{equation}\label{eq:Khatli}
 	\lim_{n\rightarrow\infty}\det\left.\big(1+\hat K_s\big)\right|_{L^2(\Sigma_{1}^{(n)}\cup\Sigma_2^{(n)},\frac{d\zeta}{2\pi\ii})}
=\det\left.\big(1- K^{(\infty)}_\x\big)\right|_{L^2(\Sigma_1^{(\infty)},\frac{d\zeta}{2\pi\ii})} ,
\end{equation}
where $K^{(\infty)}_\x= K_{12}^{(\infty)} K_{21}^{(\infty)}$ of which the kernel  is
\begin{equation}
\label{z3-25-2}
  K^{(\infty)}_\x(\xi,\eta):=e^{m_\x(\xi)+m_\x(\eta)}\int_{\Sigma^{(\infty)}_2}
\frac{e^{-2m_\x(\zeta)}}{(\xi-\zeta)(\eta-\zeta)}\frac{d\zeta}{2\pi\ii}.
\end{equation}
The determinant $\det(1- K^{(\infty)}_\x)$ equals the Fredholm determinant of the Airy operator. Indeed, this determinant is a conjugated version of the determinant in the paper  \cite{Tracy10asymptotics} on ASEP. If we call the operator in (33) of \cite{Tracy10asymptotics} $L_s(\eta, \eta')$,
then $K^{(\infty)}_\x(\xi,\eta)= e^{m_\x(\xi)} L_x(\xi, \eta) e^{-m_\x(\eta)}$. It was shown in page 153 in \cite{Tracy10asymptotics} that $\det(1+L_s)=\det(1-K_{Airy})_{(s,\infty)}=F(s)$.

Now, since $\lim_{n\to \infty} P_s(M)=\lim_{n\to \infty} \det(1+K)_{L^2(C_+\cup C_-)}$ by~\eqref{eq:oute}, ~\eqref{eq:Kssca} and~\eqref{eq:Khatli} implies that $P_\s(2\sqrt{n}+2^{-\frac{2}{3}}n^{-\frac{1}{6}}\x)\rightarrow F(\x)$ for all $s$.
All the estimates are uniform in $\s\in [0,1]$ and we obtain
$\Prob\left(\widthf_n<  2\sqrt{n}+2^{-2/3}n^{-1/6}\x\right)= \int_0^1 P_s(M)ds\to F(x)$.
This proves Theorem \ref{thm:1}.

\section{Symmetric simple random walks}\label{sec:SRW}


\subsection{Continuous-time symmetric simple random walks}\label{sec:SRW1}

Let $Y(t)$ be a continuous-time symmetric simple random walk.
This can also be thought of as the difference of two independent rate $1/2$ Poisson processes. The transition probability is given by $p_t(x,y)=p_t(y-x)$ where
\begin{equation}\label{eq:ctstr}
  p_t(k)=e^{-t}\sum_{n\in\Z}\frac{(t/2)^{2n+k}}{n!(n+k)!},\qquad k\in\Z.
\end{equation}
where $\frac1{k!}:=0$ for $k<0$ by definition.
Let $Y_i(t)$ be independent copies of $Y$ and set $X_i(t)=Y_i(t)+i$, $i=0,1,2,\cdots,n-1$.
Also set
$X(t):=(X_0(t),X_1(t),\cdots,X_{n-1}(t))$.
Then $X(0)= \left(0,1,\cdots,n-1\right)$.
We condition on the event that (a) $X(T)=X(0)$ and
(b) 
$X_0(t)<X_1(t)<\cdots< X_{n-1}(t)$ for all $t\in [0,T]$.
See, for example, \cite{AdlerFvM}.
We use the notation $\Prob$ to denote this conditional probability.

Define the `width' as
\begin{equation}
  	W_n(T) = \sup_{t\in [0,T]} (X_{n-1}(t)-X_1(t)).
\end{equation}
The analogue of Proposition~\ref{prop:BMHank} is the following. The proof is given at the end of this section.

\begin{prop}\label{prop:contiSRW}
For  non-intersecting continuous-time symmetric simple random walks,
\begin{equation}
\label{z3-27-3}
  	\Prob(\widthf_n(T)<\width)= \frac1{\ToC_n(f)} \oint_{|\s|=1}
	\To_n(f, \nodes_s)
	\frac{d\s}{2\pi\ii\s}, \quad f(z)=e^{\frac{T}2 (z+z^{-1})},
\end{equation}
and $\nodes_s=\{z\in\C : z^M=s\}$.
\end{prop}

The limit theorem is:

\begin{thm}
\label{z12-20-1}
For each $\x\in\R$,
\begin{equation}
    \lim_{\min\{n, T\}\to \infty} \Prob \left( \frac{\widthf_n(T) -\mu(n,T) }{\sigma(n,T)} \le \x \right) = F(\x)
\end{equation}
where
\begin{equation}\label{eq:msca}
    \mu(n,T):= \begin{cases}
    2\sqrt{nT}, \qquad & n<T,\\
    n+T, \qquad &n\ge T,
    \end{cases}
\end{equation}
and
\begin{equation}\label{eq:msca2}
    \sigma(n,T):= \begin{cases}
    2^{-2/3}T^{1/3}\big(  \sqrt{\frac{n}{T}}+\sqrt{\frac{T}{n}}  \big)^{1/3}, \quad & n<T,\\
    2^{-1/3}T^{1/3}, \quad &n\ge T.
    \end{cases}
\end{equation}
\end{thm}

Note that due to the initial condition and the fact that at most one of $X_j$'s moves with probability $1$ at any given time, if $X_i$ is to move downward at time $t$, it is necessary that $X_0, \cdots, X_{i-1}$ should have moved downward at least once during the time interval $[0,t)$.
Thus, if $T$ is small compared to $n$, then only a few bottom walkers can move downard (and similarly, only a few top walkers can move upward), and hence the middle walkers are `frozen'(See Figure~\ref{fig:Frozen region}). 
On the other hand, if $T$ is large compared to $n$, then there is no frozen region. 
The above result shows that the transition occurs when $T=n$ at which point the scalings~\eqref{eq:msca} and~\eqref{eq:msca2} change.

\begin{figure}[htbp]
  \centering
  \includegraphics[scale=0.65]{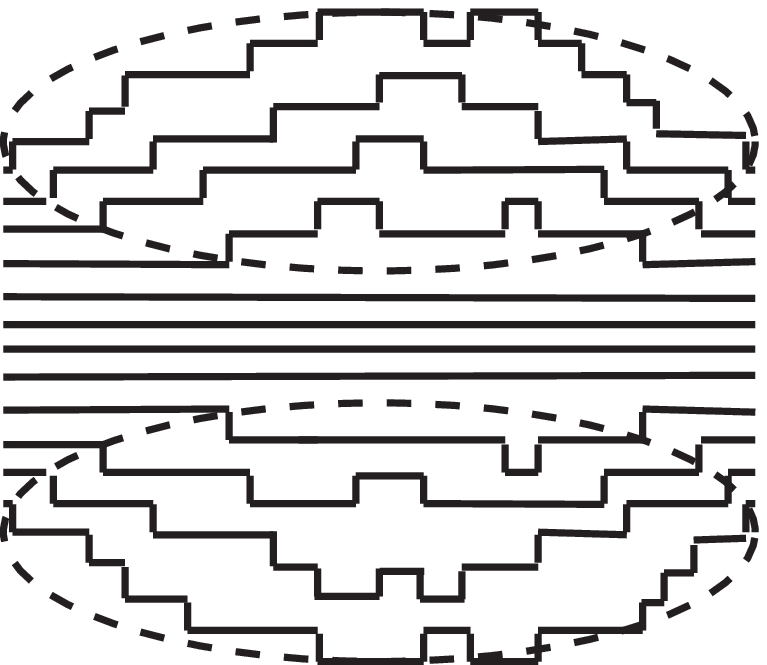}
\caption{Frozen region when $T<n$}\label{fig:Frozen region}
\end{figure}

Using Theorem~\ref{thm:0}, Theorem~\ref{z12-20-1} can be obtained following the similar analysis as in Section~\ref{sec:pfTh12} once we have the asymptotics of the (continuous) orthonormal polynomials with respect to the measure $e^{\frac{T}2(z+z^{-1})}\frac{dz}{2\pi iz}$ on the unit circle.
The asymptotics of these particular orthonormal polynomials were studied in \cite{Baik99} and \cite{Baik08} using the Deift-Zhou steepest-descent analysis of Riemann-Hilbert problems.
In order to be able to control the operator~\eqref{eq:KforToe}, the estimates on the error terms in the asymptotics need to be improved.
It is not difficult to achieve such estimates by keeping track of the error terms more carefully in the analysis of  \cite{Baik99} and \cite{Baik08}. We do not provide any details.
Instead we only comment that the difference of the scalings for $n<T$ and $n>T$ is natural from the Riemann-Hilbert analysis of the orthonormal polynomials. If we consider the orthonormal polynomial of degree $n$, $p_n(z)$, with weight $e^{\frac{T}2 (z+z^{-1})}$, the support of the equilibrium measure changes from the full circle when $\frac{n}{T}>1$ to an arc when $\frac{n}{T}<1$. The ``gap" in the support starts to appear at the point $z=-1$ when $n=T$ and grows as $\frac{n}{T}$ decreases. This results in different asymptotic formulas of the orthonormal polynomials in two different regimes of parameters.
However, we point out that the main contribution to the kernel~\eqref{eq:KforToe} turns out to come from the other point on the circle, namely $z=1$.

For technical reasons, the Riemann-Hilbert analysis is done separately for the following four overlapping regimes of the parameters: (\uppercase\expandafter{\romannumeral 1}) $n\ge T+C_1T^{1/3}$, (\uppercase\expandafter{\romannumeral 2}) $T-C_2T^{1/3}\le n\le T+C_3T^{1/3}$, (\uppercase\expandafter{\romannumeral 3}) $c_1T\le n\le T-C_4T^{1/3}$, (\uppercase\expandafter{\romannumeral 4}) $n\le c_2T$ where $0<c_k<1$ and $C_k>0$.

Here we only indicate how the leading order calculation leads to the GUE Tracy-Widom distribution for the case (I). 
We take
\begin{equation}\label{eq:4M}
	M= n+T + 2^{-1/3}T^{1/3} x.
\end{equation}
Let $p_n(z)$ be the orthonormal polynomial and $\kappa_n$ be its leading coefficient.
For case (I), the Riemann-Hilbert analysis implies that
\begin{equation}
    \kappa_n^{-1} p_n(z)\approx \begin{cases}
      z^ne^{-\frac{T}{2}z^{-1}} , \qquad &|z|>1,\\
      o(e^{-\frac{T}{2}z}) , \qquad & |z|<1,
    \end{cases}
\end{equation}
and
\begin{equation}
 	\kappa_n p_n^*(z)\approx \begin{cases}
      o(z^ne^{-\frac{T}{2}z^{-1}}) , \qquad &|z|>1,\\
     	e^{-\frac{T}{2}z} , \qquad & |z|<1.
    \end{cases}
\end{equation}
Here these asymptotics can be made uniform for $|z-1|\ge O(T^{-1/3})$. In the below, we always assume that $z$ and $w$ satisfy this condition even if we do not state it explicitly. 
The above estimates imply that the leading order of 
$z^{n/2}K_{conti}w^{-n/2}$, where $K_{conti}$ is defined in~\eqref{eq:Kconti}, becomes
\begin{equation}
 	z^{n/2}K_{conti}(z,w)w^{-n/2} \approx \begin{cases}
      	z^{n/2}\frac{e^{-\frac{T}2(z^{-1}+w)}}{1-z^{-1}w}w^{-n/2}, \quad &|z|>1, \, |w|<1,\\
     	-z^{-n/2}\frac{e^{-\frac{T}2(z+w^{-1})}}{1-z^{-1}w}w^{n/2}, \quad & |z|<1, \,  |w|>1.
    \end{cases}
\end{equation}
The kernel is of smaller order than the above when $|z|<1, |w|<1$ or $|z|>1, |w|>1$.
Since $\nodes=\nodes_s=\{z\in \C: z^M=s\}$, we choose $\gamma(z)=z^M-s$ and
\begin{equation}
  	v(z):=\begin{cases}
    	\frac{s}{z^M-s}\approx sz^{-M},\quad &|z|>1,\\
    	\frac{z^M}{s-z^M}\approx \frac{1}{s}z^M ,\quad &|z|<1.
  \end{cases}
\end{equation}
Here again the approximation is uniform for $|z-1|\ge O(T^{-1/3})$.
Hence inserting $f(z)=e^{\frac{T}2(z+z^{-1})}$, we find that the leading order term of~\eqref{eq:KforToe} is
\begin{equation}\label{eq:4Kz}
  	z^{n/2}K(z,w)w^{-n/2}
	\approx
	\pm \frac{e^{\pm (\phi(z)-\phi(w)})}{1-z^{-1}w},
	\quad \phi(z):= \frac{T}4 (z-z^{-1})- \frac{M-n}2 \log z
\end{equation}
where the sign is $+$ is  when $|z|>1, |w|<1$ and is $-$ when $|z|<1, |w|>1$. Using~\eqref{eq:4M}, we note that
\begin{equation}
  	\phi(z)
	= -\frac{T^{1/3}}{2^{4/3}} x(z-1) + \frac{T}{12}(z-1)^3
	+ O(T^{1/3}(z-1)^2)+O(T(z-1)^4).
\end{equation}
Hence for $\zeta=O(1)$,
\begin{equation}
  	\phi(1+ \frac{2^{1/3}}{T^{1/3}} \zeta)
	= -\frac12 x \zeta + \frac16 \zeta^3 + O(T^{-1/3}).
\end{equation}
After the scaling $z=1+ \frac{2^{1/3}}{T^{1/3}} \zeta$ and $w= 1+ \frac{2^{1/3}}{T^{1/3}} \eta$,~\eqref{eq:4Kz} converges to the leading term of~\eqref{eq:leadK}, except for the overall sign change which is due to the reverse orientation of the contour.
Thus we end up with the same limit~\eqref{eq:Khatli} which is $F(x)$.

\begin{proof}[Proof of Proposition~\ref{prop:contiSRW}]
Similarly to Lemma~\ref{z10-31-1} we apply the Karlin-McGregor argument in the chamber $\{x_0<x_1<\cdots<x_{n-1}<x_0+\width\}$ and obtain that
\begin{equation}
\label{z11-5-1}
  \Prob\left(\widthf_n(T)<\width\right)=\frac{  \sum_{h_j\in \Z, h_0+h_1+\cdots+h_{n-1}=0}
  \det\left[ p_T(x_j-y_k+h_k\width)\right]_{j,k=0}^{n-1}}{\det\left[p_T(x_j-y_k)\right]_{j,k=0}^{n-1}}.
  \end{equation}
  Note that numerator equals
  \begin{equation}
    \oint_{|s|=1}\det\left[\sum_{h\in\Z}p_T(x_j-y_k+h\width)s^h\right]_{j,k=0}^{n-1}\frac{ds}{2\pi is}.
  \end{equation}
  Since~\eqref{eq:ctstr} can be written as
  \begin{equation}
  p_T(x)=e^{-T}\oint_{|z|=1}z^{-x}e^{\frac{T}{2}(z+z^{-1})}\frac{dz}{2\pi iz},
  \end{equation}
  we find that 
  \begin{equation}
    \sum_{h\in\Z}p_T(x+h\width)s^h=\frac{e^{-T}}{\width}\sum_{z^\width=s}z^{-x}e^{\frac{T}{2}(z+z^{-1})}.
  \end{equation}
  Proposition~\ref{prop:contiSRW} follows immediately.
\end{proof}


\begin{rmk}\label{rmk:sec4}
If we were to evaluate the ratio $\frac{\To_n(f, \nodes_s)}{\ToC_n(f)}$ directly instead of using the Fredholm determinant formula, we need to find the asymptotic expansion of
the log of the determinants to the order $o(1)$ including the constant term.
This is relatively easy to obtain for $\ToC_n(f)$ when $\frac{T}{n}<1$: the Szeg\"o limit theorem essentially applies with an exponentially decaying error term.
However, when $\frac{T}{n}>1$, this calculation is cumbersome and complicated \cite{Baik99}, and the asymptotic expansions had not been obtained to the desired order .
Especially, the determination of the constant term in the asymptotic expansion would require some sophisticated analysis  (see e.g. \cite{DIK, Baik08}). 
The difficulty is due to the following fact that the orthogonal polynomials only give the asymptotics of the ratio $\ToC_k(f)/\ToC_{k-1}(f)$, whose error terms are of exponential type when $\frac{T}{n}<1$ but are of polynomial type when $\frac{T}{n}>1$.
This technicality is also directly related to the difficulty in obtaining the precise asymptotic in the lower tail regime for the length of the longest increasing subsequences or other directed last passage percolation models
\cite{Baik99, BDMPZ}.
For $f$ above, it turns out that the discrete Toeplitz determinant $\To_n(f, \nodes_m)$
essentially factors into two parts asymptotically, one of which is same as the asymptotic of the continuous Toeplitz determinant \cite{BaikJenkins}.
The formula~\eqref{z3-17-1} is precisely of the form that this cancellation is already taken into account.
By this reason, we could evaluate the limit of $\frac{\To_n(f, \nodes_s)}{\ToC_n(f)}$ for certain $m$ even if we do not have the asymptotic formula of each determinant to the order $o(1)$.
We note that the asymptotic evaluation of the Fredholm determinant may become difficult for other choices of $m$, especially for those which correspond to the so-called `saturated region' conditions for the discrete orthogonal polynomials.
\end{rmk}

\subsection{Discrete-time symmetric simple random walks}

Let $X_0(k), \cdots, X_{n-1}(k)$, $k=0,1,\cdots, n-1$, be independent  discrete-time symmetric simple random walks. Set
$X(k):=(X_0(k),X_1(k),\cdots,X_{n-1}(k))$.
We take the initial condition as
\begin{equation}
    X(0)=(0,2,\cdots,2n-2).
\end{equation}
and consider the process conditional of the event that (a) $X(2T)=X(0)$ and
(b) $X_0(k)<X_1(k)<\cdots< X_{n-1}(k)$ for all $k=0,1, \cdots, 2T$.
The non-intersecting discrete-time simple random walks can also be interpreted as random tiling of a hexagon and were studied in many papers. See, for example, \cite{CohnLarsenPropp98,Johansson03,BKMMbook,BorodinGorin09}.
The notation $\Prob$ denotes this conditional probability.
Define the width $\widthf_n(2T):= \max_{k=0,1,\cdots, 2T} \big( X_{n-1}(k)- X_0(k) \big)$ as before.

\begin{prop}\label{prop:SRW}
For  non-intersecting discrete-time symmetric simple random walks,
\begin{equation}\label{z3-27-2}
  	\Prob(\widthf_n(2T)<2\width)= \frac1{\ToC_n(f)} \oint_{|\s|=1}
	\To_n(f, \nodes_s)
	\frac{d\s}{2\pi\ii\s}, \quad f(z)=z^{-T}(1+z)^{2T},
\end{equation}
and $\nodes_s=\{z\in\C : z^M=s\}$.
\end{prop}

The fluctuations are again given by $F$. Note that $2n\le \widthf_n(2T)\le 2n+2T$ for all $n$ and $T$.

\begin{thm}
\label{z3-27-1}
Fix $\gamma>0$ and $0<\beta<2$. Then for $n=[\gamma T^\beta]$,
\begin{equation}
        \lim_{T\to\infty} \Prob \left( \frac{\widthf_n(2T) -2\sqrt{n^2+2nT}}{(n^2+2nT)^{-\frac{1}{6}}T^{\frac{2}{3}}} \le\x \right)=
    F(\x).
  \end{equation}
for each $\x\in \R$.
\end{thm}

Note that the parameter $(n^2+2nT)^{-\frac{1}{6}}T^{\frac{2}{3}} \to \infty$ as $T\to \infty$ when $\beta<2$.
This parameter is $O(1)$ when $\beta=2$. Indeed one can show that when $\beta>2$,
\begin{equation}
    \lim_{T\to\infty} \Prob(\widthf_n(2T)=2n+2T)= 1.
\end{equation}

The proofs of the proposition and the theorem are similar to those for the continuous-time symmetric simple random walks and we omit them. 


\section{Proof of Theorem~\ref{z11-3-4}}\label{sec:LPP}

In this section we give a proof of Theorem~\ref{z11-3-4}.
The proof is based on the results on a solvable directed  last passage percolation model and is similar to  the proof of the identity~\eqref{eq:GOE} by Johansson~\cite{Johansson04}.


By symmetry we may assume $\alpha\le\beta$.
Let $w(i,j),(i,j)\in\N^2,$ be independent random variables with geometric distribution,
$\Prob(w(i,j)=k)=(1-q)q^k$, $k=0,1,2,\cdots$.
Define the random variable (point-to-point directed last passage time)
\begin{equation}\label{eq:Gde}
  G(M,N)=\max_{\pi} \bigg( \sum_{(i,j)\in\pi}w(i,j) \bigg),
\end{equation}
where the maximum is taken over all possible up/right paths from $(1,1)$ to $(M,N)$.
The limiting fluctuations of $G(M,N)$ are known to be $F$ in \cite{johansson00} as $M$ and $N$ tend to infinite with a finite ratio. In particular, when $M=N=(\alpha+\beta)n$, 
\begin{equation}
\label{z11-11-5}
  \lim_{n\rightarrow\infty}\Prob\left(\frac{G((\alpha+\beta)n,(\alpha+\beta)n)-\mu (\alpha+\beta)n}{\sigma (\alpha+\beta)^{1/3}n^{1/3}}\le s\right)=F(s),
\end{equation}
where
\begin{equation}\label{eq:sig}
  \mu=\frac{2\sqrt{q}}{1-\sqrt{q}},\qquad \sigma=\frac{q^{1/6}(1+\sqrt{q})^{1/3}}{1-\sqrt{q}}.
\end{equation}

Consider the lattice points on the line connecting the points $(1,2\alpha n)$ and $(2\alpha n,1)$, i.e.
$\calL:=\{(\alpha n+u, \alpha n-u): |u|<\alpha n\}$.
An up/right path from $(1,1)$ to $((\alpha+\beta)n,(\alpha+\beta)n)$ passes through a point on $\calL$.
Considering the up/right path from $(1,1)$ to a point on $\calL$ and the down/left path from $((\alpha+\beta)n,(\alpha+\beta)n)$ to the same point on $\calL$ (see Figure~\ref{fig:LPP}),
we find that $G((\alpha+\beta)n, (\alpha+\beta)n)$ equals
\begin{equation}
\label{z11-11-6}
  \max_{|u|<\alpha n}\left(G^{(1)}(\alpha n+u,\alpha n-u)+G^{(2)}(\beta n+u,\beta n-u)\right)+O(1),
\end{equation}
where $G^{(1)}$ and $G^{(2)}$ are two independent copies of $G$, and the error term $O(1)$ comes from the duplicate diagonal term $w(\alpha n+u,\alpha n-u)$.

\begin{figure}[htbp]
  \centering
  \includegraphics[scale=0.35]{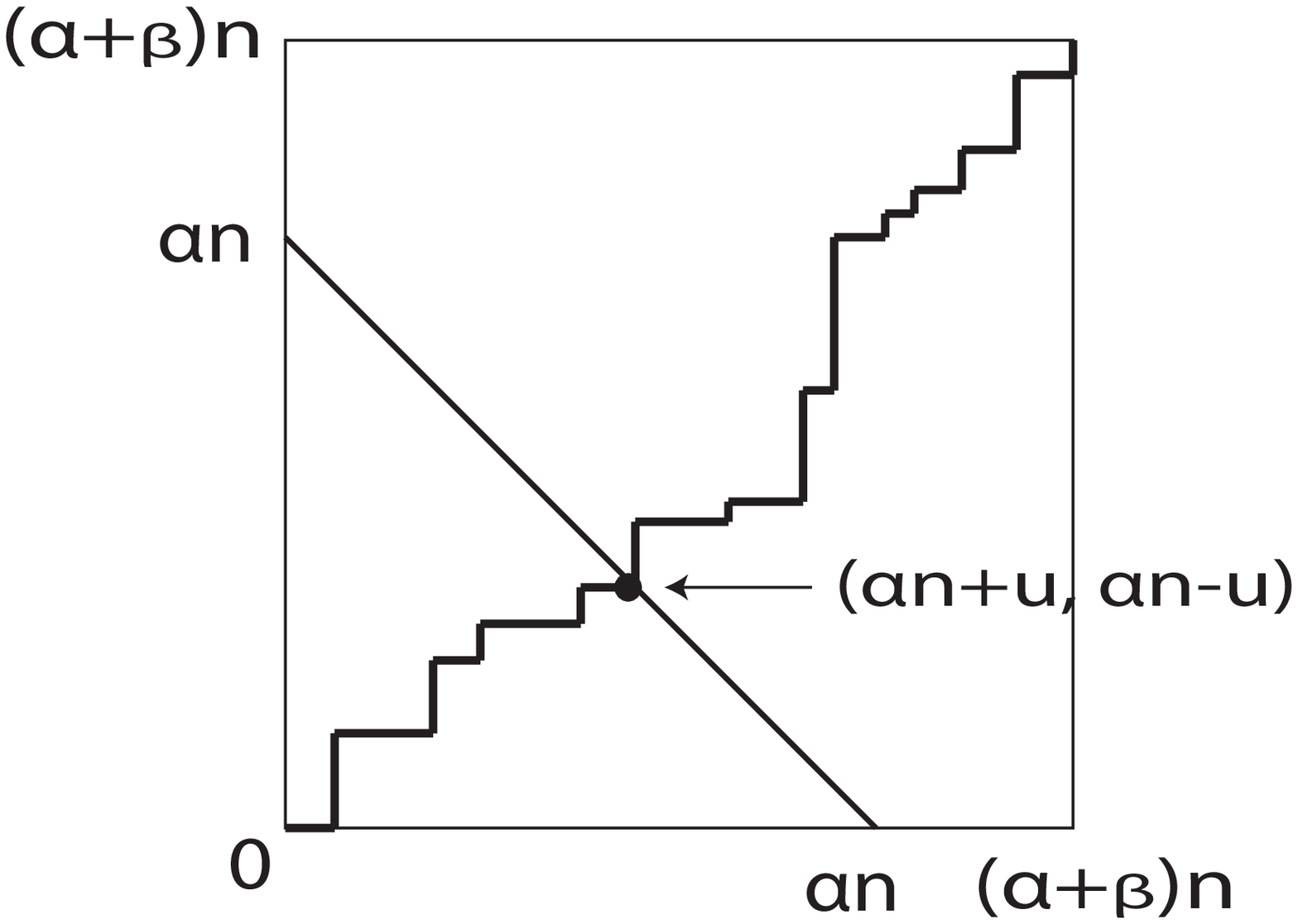}
  \caption{Intersection of an up/right path with $\mathcal{L}$}
  \label{fig:LPP}
\end{figure}

Consider $G^{(i)}(\alpha n+u,\alpha n-u)$ as a process in time $u$.
For $u$ of order $n^{2/3}$, it was shown in \cite{Johansson04} that the fluctuations of this process converge the Airy process in the functional convergence. More precisely, if we set
\begin{equation}
  	H^{(1)}_n(\tau):=\frac{G^{(1)}(\alpha n+d^{-1}(\alpha n)^{2/3}\tau,\alpha n-d^{-1}(\alpha n)^{2/3}\tau)-\mu \alpha n}{\sigma (\alpha n)^{1/3}},
\end{equation}
and
\begin{equation}
  	H^{(2)}_n(\tau):=\frac{G^{(2)}(\beta n+d^{-1}(\beta n)^{2/3}\tau,\beta n-d^{-1}(\beta n)^{2/3}\tau)-\mu \beta n}{\sigma (\beta n)^{1/3}},
\end{equation}
for $|\tau|<d(\alpha n)^{1/3}$, where $d:=q^{1/6}(1+\sqrt{q})^{-2/3}$, then $H^{(i)}_n(\tau)$ converges to the Airy process $\mathcal{A}^{(i)}(\tau)-\tau^2$, $i=1,2$.
(We note that there is a typographical error in the formula (1.8) in \cite{Johansson04} where, in terms of our notations, $\sigma$ is changed to $\frac{q^{1/6}(1+\sqrt{q})^{1/3}}{1-q}$. However, the correct formula of $\sigma$ is $\frac{q^{1/6}(1+\sqrt{q})^{1/3}}{1-\sqrt{q}}$ as in~\eqref{eq:sig} which is also  same as in \cite{johansson00}.)
Since~\eqref{z11-11-6} implies that
\begin{equation}
\label{z11-21-1}
\begin{split}
  &\Prob\left(\frac{G(N,N)-\mu N}{\sigma N^{1/3}}\le s\right)\\
  &= \Prob\left(\max_{|\tau|< d\alpha n^{1/3}}\left(\alpha^{1/3}H_n^{(1)}(\alpha^{-2/3}\tau)+\beta^{1/3}H_n^{(2)}(\beta^{-2/3}\tau)\right)\le (\alpha+\beta)^{1/3}s\right)+O(N^{-1/3}),
\end{split}
\end{equation}
we obtain Theorem~\ref{z11-3-4} if we prove that
\begin{equation}
  \label{z11-21-2}
  \begin{split}
  	&\lim_{n\to\infty}  \Prob\left(\max_{|\tau|< d\alpha n^{1/3}}\left(\alpha^{1/3}H_n^{(1)}(\alpha^{-2/3}\tau)+\beta^{1/3}H_n^{(2)}(\beta^{-2/3}\tau)\right)\le (\alpha+\beta)^{1/3}s\right)\\
	&=\Prob \left( \max_{\tau\in \R} \left(\alpha^{1/3}\mathcal{A}^{(1)}(\alpha^{-2/3}\tau)
    +\beta^{1/3}\mathcal{A}^{(2)}(\beta^{-2/3}\tau)-(\alpha^{-1}+\beta^{-1})\tau^2 \right) \le (\alpha+\beta)^{1/3}s \right).
\end{split}
\end{equation}
In \cite{Johansson04}, a similar identity
\begin{equation}
\label{z11-21-3}
  \lim_{n\to\infty}\Prob\left(\max_{|\tau|< dn^{1/3}} H_n(\tau)\le s\right)=\Prob\left(\max_{\tau\in\R}(\mathcal A(\tau)-\tau^2)\le s\right)
\end{equation}
was proved as a part of the proof of~\eqref{eq:GOE}.
We proceed similarly and use the estimates obtained in \cite{Johansson04} .

Set
\begin{equation}
\begin{split}
	X_{n,T}:= (\alpha+\beta)^{-1/3}\cdot \max_{|\tau|\le T}\left(\alpha^{1/3}H_n^{(1)}(\alpha^{-2/3}\tau)+\beta^{1/3}H_n^{(2)}(\beta^{-2/3}\tau)\right)
\end{split}
\end{equation}
and
\begin{equation}
\begin{split}
	Y_{n,T}:= (\alpha+\beta)^{-1/3}\cdot \max_{|\tau|>T}\left(\alpha^{1/3}H_n^{(1)}(\alpha^{-2/3}\tau)+\beta^{1/3}H_n^{(2)}(\beta^{-2/3}\tau)\right).
\end{split}
\end{equation}
Since
\begin{equation}
\begin{split}
	&\Prob(X_{n,T}\le s)\\
	&\ge \Prob\left(\max_{|\tau|< d\alpha n^{1/3}}\left(\alpha^{1/3}H^{(1)}(\alpha^{-2/3}\tau)+\beta^{1/3}H_n^{(2)}(\beta^{-2/3}\tau)\right)\le (\alpha+\beta)^{1/3}s\right)\\
	&\ge \Prob( X_{n,T} \le s) - \Prob( Y_{n,T} > s)
\end{split}
\end{equation}
for all large enough $n$ for each fixed $T$,~\eqref{z11-21-2} follows from the following three properties:

\begin{enumerate}[(a)]
\item For each $\epsilon>0$, there are positive constants $T_0$ and $n_0$ such that
$\Prob\left(Y_{n,T}> s\right)<\epsilon$
for all $T>T_0$ and $n>n_0$,

\item For each fixed $T$, $\Prob\left( X_{n,T}\le  s\right)\to \Prob \left( A_T \le s \right)$ as $n\to \infty$.

\item Finally, $\Prob(A_T \le s) \to \Prob(A_\infty \le s)$ as $T\to \infty$.
\end{enumerate}

Here
\begin{equation}
\begin{split}
	A_T:= (\alpha+\beta)^{-1/3}\cdot \max_{|\tau|\le T}\left(\alpha^{1/3}\mathcal A^{(1)}(\alpha^{-2/3}\tau)+\beta^{1/3}\mathcal A^{(2)}(\beta^{-2/3}\tau)-(\alpha^{-1}+\beta^{-1})\tau^2\right)
\end{split}
\end{equation}
and  $A_{\infty}$ is the same random variable with the maximum taken over $\tau\in \R$.


A functional limit theorem to the Airy process was proved in \cite{Johansson04} (Theorem 1.2). This means that $H_n^{(i)}(\tau) \to\mathcal A^{(i)}(\tau)-\tau^2$ at $n\to \infty$ in the sense of weak convergence of the probability measures on $C[-T,T]$ for each fixed $T$.
Hence the property (b) follows
a theorem on the convergence of product measures (\cite{BillingsleyBook}, Theorem 3.2).

The property (c) follows from the monotone convergence theorem since $\{A_\infty \le s\} = \cap_{T>0} \{A_T\le s\}$.


For the property (a), we use the estimates (5.19) and (5.20) in \cite{Johansson04}: there are positive constants $C$ and $c$ such that
\begin{equation}
\begin{split}
  	&\Prob\left(\max_{T<\tau \le \log n} H_n^{(i)}(\alpha^{-2/3}\tau)>M \right)\\
  &\le\int_{\alpha^{-2/3} T-1}^{\infty}e^{-c(M-1+x^2)^{3/2}}dx+C\int_{\alpha^{-2/3} T-1}^{\infty}e^{-x^3}dx
\end{split}
\end{equation}
and
\begin{equation}
  	\Prob\left(\max_{\tau \ge \log n}H_n^{(i)}(\alpha^{-2/3}\tau)>M\right)\le Cne^{-c(\log n)^3}
\end{equation}
for all $M$.
Therefore, taking $M=\alpha^{-1/3}(\alpha+\beta)^{1/3}s/2$, for any $\epsilon>0$, we have
\begin{equation}
	\Prob\left((\alpha+\beta)^{-1/3} \max_{\tau \ge T}\alpha^{1/3}H_n^{(i)}(\alpha^{-2/3}\tau)>\frac{s}{2}\right)<\frac{\epsilon}{2},
\end{equation}
if $T,n$ are both large enough.
This proves (a).


\end{document}